\documentclass[a4paper]{amsart}

\usepackage[T1]{fontenc}
\usepackage{lmodern}
\usepackage{amsmath, amssymb, amsfonts, amsthm}
\usepackage{mathrsfs}
\usepackage{graphicx}
\usepackage{enumerate}
\usepackage{tikz}
\usepackage{url}
\usepackage[colorlinks=true, allcolors=black]{hyperref}

\makeatletter
\@addtoreset{equation}{section}
\makeatother

\theoremstyle{plain}
\newtheorem{theorem}{Theorem}[section]
\newtheorem{lemma}[theorem]{Lemma}

\newtheorem{question}[theorem]{Question}
\newtheorem{maintheorem}{Theorem}

\newtheorem{maincorollary}{Corollary}

\newtheorem{mainlemma}{Lemma}

\theoremstyle{definition}
\newtheorem{definition}[theorem]{Definition}

\theoremstyle{remark}
\newtheorem{remark}[theorem]{Remark}

\begin{document}

\title[Bohr chaoticity, semi-horseshoes and full-entropy abundance]{Bohr chaoticity, semi-horseshoes and full-entropy abundance}

\author{Xiaobo Hou, Wanshan Lin and Xueting Tian}

\address{Xiaobo Hou, School of Mathematical Sciences, Dalian University of Technology, Dalian 116024, P.R. China}
\email{xiaobohou@dlut.edu.cn}

\address{Wanshan Lin, School of Mathematical Sciences,  Fudan University, Shanghai 200433, P.R. China}
\email{21110180014@m.fudan.edu.cn}

\address{Xueting Tian, School of Mathematical Sciences,  Fudan University, Shanghai 200433, P.R. China}
\email{xuetingtian@fudan.edu.cn}

\begin{abstract}
	Bohr chaoticity is a topological notion of dynamical complexity defined through non-orthogonality to all non-trivial weights. It is strictly stronger than positivity of topological entropy and also has strong consequences for the invariant-measure structure. In this paper, we show that every dynamical system having a semi-horseshoe, including every positive-entropy graph map and every $C^1$ partially hyperbolic diffeomorphism, is Bohr chaotic; furthermore, the set of points correlated with any given non-trivial weight has positive topological entropy. Moreover, for positive-entropy dynamical systems with either the shadowing property or the modified almost specification property, such set can has full topological entropy. Our results also yield applications in several classical algebraic and smooth settings, as well as in the $C^0$-generic setting of topological dynamics.
\end{abstract}

\thanks{All authors are co-first authors of the article.}
\keywords{Bohr chaoticity, semi-horseshoe, shadowing property, specification-like property, topological entropy}
\subjclass[2020] {37B10; 37B40; 37D20; 37D25}
\maketitle
\section{Introduction}

Throughout this paper, $(X,f)$ denotes a topological dynamical system, where $X$ is a compact metric space and $f:X\to X$ is continuous. Let $\mathcal M(X)$, $\mathcal M_f(X)$, and $\mathcal M_f^e(X)$ denote the spaces of Borel probability measures, $f$-invariant Borel probability measures, and $f$-ergodic Borel probability measures on $X$, respectively. We write $\mathbb Z$, $\mathbb N$, and $\mathbb N^+$ for the integers, non-negative integers, and positive integers. Let $C(X)$ denote the space of real-valued continuous functions on $X$, endowed with the norm
\[
\|\varphi\|:=\sup_{x\in X}|\varphi(x)|.
\]
In a Baire space, a set is said to be \emph{residual} if it contains a dense $G_\delta$ subset.

Weighted orthogonality in topological dynamics is closely related to Sarnak's conjecture, which predicts that the M\"obius function is orthogonal to every zero-entropy topological dynamical system \cite{Sarnak}. In this direction, Bohr chaoticity, introduced by Fan, Fan, Ryzhikov and Shen \cite{Fan-Fan-Ryzhikov-Shen2022} and later extended from $\mathbb Z$- to $\mathbb Z^d$-actions by Fan, Schmidt and Verbitskiy \cite{Fan-Schmidt-Verbitskiy-2024}, may be viewed as a strong opposite phenomenon: it requires non-orthogonality to every non-trivial weight. More precisely, a topological dynamical system $(X,f)$ is said to be \emph{Bohr chaotic} if for every non-trivial weight $\vartheta=(\vartheta_n)_{n\ge 0}$, there exist $\varphi\in C(X)$ and $x\in X$ such that
\[
\limsup_{N\to\infty}\frac{1}{N}
\left|
\sum_{n=0}^{N-1}\vartheta_n\varphi(f^n x)
\right|>0,
\]
where a bounded real sequence
\[
\vartheta=(\vartheta_n)_{n\ge 0}\in \ell^\infty(\mathbb N)
\]
is called a \emph{non-trivial weight} if
\[
\limsup_{N\to\infty}\frac{1}{N}\sum_{n=0}^{N-1}|\vartheta_n|>0.
\]
Equivalently, $(X,f)$ is Bohr chaotic if no non-trivial weight is orthogonal to it, where a non-trivial weight $\vartheta$ is said to be \emph{orthogonal} to $(X,f)$ if for every $\varphi\in C(X)$ and every $x\in X$,
\[
\lim_{N\to\infty}\frac{1}{N}
\sum_{n=0}^{N-1}\vartheta_n\varphi(f^n x)=0.
\]
We work here with real-valued continuous functions and real weights; this is equivalent to the complex-valued version, see \cite[p.~348]{Tal2023}.

Bohr chaoticity is closely related to entropy, but it is strictly stronger than mere positivity of topological entropy \cite[p.~1128]{Fan-Fan-Ryzhikov-Shen2022}. It also has strong consequences for the invariant-measure structure: by \cite[Corollary 2.5]{Tal2023}, a Bohr chaotic system must have continuum many ergodic invariant measures. Existing sufficient conditions for Bohr chaoticity arise mainly from horseshoe-type structures \cite{Fan-Fan-Ryzhikov-Shen2022}, from specification property \cite{Tal2023}, and from certain algebraic models \cite{Fan-Fan-Ryzhikov-Shen2022,Fan-Schmidt-Verbitskiy-2024}. These results reveal Bohr chaoticity as a robust manifestation of orbit complexity, but a flexible mechanism covering broader orbit-tracing phenomena is still missing. Related shadowing-based sufficient conditions for Bohr chaos were also obtained by Kawaguchi in connection with hyperbolic sets \cite{Kawaguchi-2026}.

\subsection{Semi-horseshoes imply Bohr chaoticity}
Let $(\{0,1\}^{\mathbb{N}},\sigma)$ and $(\{0,1\}^{\mathbb{Z}},\sigma)$ be the one-sided full shift and two-sided full shift, respectively. Given $N\in\mathbb{N}^+$, recall that a subsystem $(\Lambda,f^N)$ of $(X,f^N)$ is said to be a \emph{one-sided $N$-order semi-horseshoe} (resp. \emph{two-sided $N$-order semi-horseshoe}) of $(X,f)$ if there exists a continuous surjection $$\pi:\Lambda\to\{0,1\}^{\mathbb{N}}~(\text{resp. }\{0,1\}^{\mathbb{Z}})$$ such that $$\pi\circ f^N=\sigma\circ\pi.$$ If further $\pi$ is injective, then we say that $(\Lambda,f^N)$ is a \emph{one-sided $N$-order horseshoe} (resp. \emph{two-sided $N$-order horseshoe}) of $(X,f)$. For convenience, we say that $(X,f)$ has a horseshoe (resp. semi-horseshoe) if it has a one-sided $N$-order horseshoe (resp. semi-horseshoe) or a two-sided $N$-order horseshoe (resp. semi-horseshoe) for some $N\in\mathbb{N}^+$. Fan, Fan, Ryzhikov and Shen \cite{Fan-Fan-Ryzhikov-Shen2022} showed that every dynamical system having a horseshoe is Bohr chaotic. Moreover, they constructed a partially hyperbolic toral automorphism having a semi-horseshoe but no horseshoes, and posed the following question.
\begin{question}\cite[Question 1]{Fan-Fan-Ryzhikov-Shen2022}
	Is every dynamical system having a semi-horseshoe Bohr chaotic?
\end{question}
In the first part of this paper, we answer this question affirmatively.
\begin{maintheorem}\label{Theorem A}
	If $(X,f)$ has a semi-horseshoe, then $(X,f)$ is Bohr chaotic.
\end{maintheorem}
The proof is entirely topological. It does not use the displacement machinery developed in \cite{Fan-Fan-Ryzhikov-Shen2022} for genuine horseshoes. Instead, the central idea is to pass from the given $N$-order semi-horseshoe to a higher-order prototype whose $N$-spaced levels are disjoint, then iteratively cut away the remaining overlap times until a first-return set appears. 

There are many dynamical systems with positive entropy having a semi-horseshoe, including graph maps \cite{Llibre-Misiurewicz-1993}, automorphisms of compact metric abelian groups \cite{Huang-Li-Xu-Ye}, $C^1$ partially hyperbolic diffeomorphisms \cite{Huang-Li-Xu-Ye}, dynamical systems with either the shadowing property \cite{DOT} or the modified almost specification property (Theorem \ref{Theorem 5.1}) and so on. For these dynamical systems, we can use Theorem \ref{Theorem A} to show that they are Bohr chaotic. In fact, we can show that the set of points correlated with any given non-trivial weight has positive topological entropy, see Theorem \ref{new-theorem8.1}.

\subsection{Full-entropy abundance}
In the second part of this paper, we consider the full-entropy abundance under two major orbit-tracing mechanisms in topological dynamics: the shadowing property and specification-like properties. More precisely, we show that positive topological entropy, together with either the shadowing property or the modified almost specification property, forces the set of points correlated with any given non-trivial weight has full topological entropy, which is substantially stronger than mere existence of a correlated orbit. For background on the shadowing property and specification-like properties, we refer to \cite{Kwietniak-Lacka-Oprocha-2016}.

For a non-trivial weight $\vartheta$, let $\mathcal{N}_\vartheta(f,X)$ denote the set of points $x\in X$ for which
\[
\limsup_{N\to\infty}\frac{1}{N}
\left|
\sum_{n=0}^{N-1}\vartheta_n\varphi(f^n x)
\right|>0
\]
for some $\varphi\in C(X)$. Thus $\mathcal{N}_\vartheta(f,X)$ is precisely the set of points correlated with the weight $\vartheta$. Given a subset $Z\subset X$, let $h_{\mathrm{top}}(f,Z)$ denote the  topological entropy of $Z$ introduced by Bowen \cite{Bowen1973}.

We formulate our results in a cyclic decomposition setting:
\[
X=\bigcup_{i=0}^{L-1}f^i(Y),
\]
where $(Y,f^L)$ has positive entropy and satisfies the relevant orbit-tracing property. This formulation allows us to treat not only systems that themselves satisfy shadowing or modified almost specification, but also situations in which these properties arise naturally after passing to a suitable iterate and invariant component.

We first state the shadowing version.

\begin{maintheorem}\label{Theorem B}
	Let $(X,f)$ be a dynamical system, and let $L\in\mathbb{N}^+$. Assume that there exists a subsystem $(Y,f^L)$ of $(X,f^L)$ such that
	\begin{enumerate}[(1)]
		\item $(Y,f^L)$ satisfies the shadowing property and $h_{\mathrm{top}}(f^L,Y)>0$;
		\item $X=\bigcup_{i=0}^{L-1}f^i(Y)$.
	\end{enumerate}
	Then for any non-trivial weight $\vartheta\in \ell^\infty(\mathbb{N})$, we have
	\[
	h_{\mathrm{top}}(f,\mathcal{N}_\vartheta(f,X))
	=
	h_{\mathrm{top}}(f,X)
	=
	\frac{1}{L}h_{\mathrm{top}}(f^L,Y)>0.
	\]
	In particular, $(X,f)$ is Bohr chaotic.
\end{maintheorem}

\begin{remark}\label{Remark 1.1}
	Under the assumptions of Theorem \ref{Theorem B}, if we further assume that
	$f^i(Y)\cap f^j(Y)=\emptyset$ for any $0\leq i<j\leq L-1$, then one can show
	that $(X,f)$ itself has the shadowing property. Without this additional
	assumption, however, we cannot guarantee that $(X,f)$ has the shadowing
	property. See Section \ref{Section 7} for details.
\end{remark}

In particular, if $(X,f)$ itself satisfies the shadowing property and $h_{\mathrm{top}}(f,X)>0$, then
\[
h_{\mathrm{top}}(f,\mathcal{N}_\vartheta(f,X))=h_{\mathrm{top}}(f,X)>0
\]
for every non-trivial weight $\vartheta\in \ell^\infty(\mathbb N)$. Thus, in the presence of shadowing, positive entropy alone is enough to force not only Bohr chaoticity but also full-entropy abundance of points correlated with $\vartheta$.

For dynamical systems satisfying specification-like properties, to treat weak specification, backward weak specification, and almost specification within a single framework, we work with the modified almost specification property; see Figure \ref{fig-1} and Section \ref{Section 2.3} for the relations with other specification-like notions and for precise definitions. In the same cyclic decomposition setting, we show that positive entropy together with the modified almost specification property on $(Y,f^L)$ also forces full-entropy abundance of points correlated with each non-trivial weight.

\begin{figure}[htbp]
	\begin{center}
		\tikzset{every picture/.style={line width=0.75pt}}
		\begin{tikzpicture}[
			>=latex,
			box/.style={
				draw,
				rounded corners=4pt,
				minimum width=3.1cm,
				minimum height=0.9cm,
				inner sep=3pt,
				align=center
			},
			arrow/.style={
				->,
				line width=0.8pt
			}
			]
			
			\node[box] (weak) at (0,0) {(forward/almost)\\weak specification};
			\node[box] (spec) at (-3.8,0) {specification};
			\node[box] (moda) at ( 3.8,0) {modified almost\\specification};
			
			\node[box] (almost) at (0, 1.6) {almost specification};
			\node[box] (bweak)  at (0,-1.6) {backward weak\\specification};
			
			\draw[arrow] (spec.east) -- (weak.west);
			\draw[arrow] (weak.east) -- (moda.west);
			
			\draw[arrow] (spec.north east) .. controls (-2.2,1.2) .. (almost.west);
			\draw[arrow] (spec.south east) .. controls (-2.2,-1.2) .. (bweak.west);
			\draw[arrow] (almost.east)     .. controls ( 2.2,1.2) .. (moda.north west);
			\draw[arrow] (bweak.east)      .. controls ( 2.2,-1.2) .. (moda.south west);
		\end{tikzpicture}
	\end{center}
	\caption{Relationships between specification-like properties}
	\label{fig-1}
\end{figure}
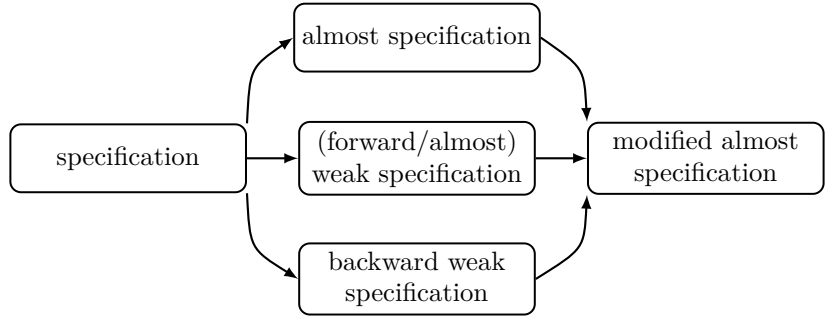

\begin{maintheorem}\label{Theorem C}
	Let $(X,f)$ be a dynamical system, and let $L\in\mathbb{N}^+$. Assume that there exists a subsystem $(Y,f^L)$ of $(X,f^L)$ such that
	\begin{enumerate}[(1)]
		\item $(Y,f^L)$ satisfies the modified almost specification property and $h_{\mathrm{top}}(f^L,Y)>0$;
		\item $X=\bigcup_{i=0}^{L-1}f^i(Y)$.
	\end{enumerate}
	Then for any non-trivial weight $\vartheta\in \ell^\infty(\mathbb{N})$, we have
	\[
	h_{\mathrm{top}}(f,\mathcal{N}_\vartheta(f,X))
	=
	h_{\mathrm{top}}(f,X)
	=
	\frac{1}{L}h_{\mathrm{top}}(f^L,Y)>0.
	\]
	In particular, $(X,f)$ is Bohr chaotic.
\end{maintheorem}

\begin{remark}\label{Remark 1.2}
	We cannot guarantee that $(X,f)$ in Theorem \ref{Theorem C} itself satisfies
	the modified almost specification property; see Theorem \ref{Theorem 6.2}.
	When $(Y,f^L)$ satisfies the specification property, it was shown in
	\cite[Theorem 4.2.3]{Tal2023} that $(X,f)$ is Bohr chaotic.
\end{remark}

Likewise, if $(X,f)$ itself satisfies the modified almost specification property and $h_{\mathrm{top}}(f,X)>0$, then the same full-entropy conclusion holds for every non-trivial weight $\vartheta\in \ell^\infty(\mathbb N)$. In this sense, Bohr chaoticity follows from positive entropy under two broad orbit-tracing mechanisms: shadowing and modified almost specification.

We now highlight three representative consequences of Theorems \ref{Theorem B} and \ref{Theorem C}. These examples show that the Bohr-chaotic abundance phenomenon established above appears naturally in algebraic dynamics, smooth dynamics, and $C^0$-generic topological dynamics. Further applications are deferred to Section \ref{Section 8}.

We first consider the algebraic setting. Let $G$ be a compact metric abelian
group and let $f$ be an automorphism of $G$. Recall that a point $x\in G$ is positively homoclinic if
\[
\lim_{n\to\infty} f^n(x)=e_G,
\]
where $e_G$ is the identity element of $G$. Denote by $\Delta_f^+(G)$ the set
of all positively homoclinic points. In the proof of \cite[Theorem 1.4]{Huang-Li-Xu-Ye}, Huang, Li, Xu and Ye proved that
$\bigl(\overline{\Delta_f^+(G)},f\bigr)$ has the (almost) weak specification property and that
\[
h_{\mathrm{top}}\bigl(f,\overline{\Delta_f^+(G)}\bigr)>0
\]
whenever $h_{\mathrm{top}}(f,G)>0$. Therefore Theorem \ref{Theorem C} yields the following
consequence.

\begin{maincorollary}
	Let $G$ be a compact metric abelian group, and let $f$ be an automorphism of
	$G$. Suppose that $h_{\mathrm{top}}(f,G)>0$. Then for every non-trivial weight
	$\vartheta\in \ell^\infty(\mathbb N)$,
	\[
	h_{\mathrm{top}}\bigl(f,\mathcal N_\vartheta(f,G)\bigr)
	\ge h_{\mathrm{top}}\bigl(f,\overline{\Delta_f^+(G)}\bigr)>0.
	\]
	In particular, $(G,f)$ is Bohr chaotic.
\end{maincorollary}

We next turn to smooth dynamics. In the $C^{1+\alpha}$ setting, Bohr chaoticity itself is already known from the existence of hyperbolic horseshoes and the horseshoe criterion for Bohr chaoticity. The novelty here is the entropy estimate for the correlated set: every hyperbolic ergodic measure with positive entropy yields a lower bound, and in the surface case this set even has full topological entropy.

\begin{maincorollary}\label{Corollary B}
	Let $M$ be a compact connected boundaryless smooth Riemannian manifold, and let
	$f:M\to M$ be a $C^{1+\alpha}$ diffeomorphism with $\alpha>0$. Suppose that
	$\mu\in \mathcal M_f^e(M)$ is hyperbolic and satisfies $h_\mu(f)>0$. Then for
	every non-trivial weight $\vartheta\in \ell^\infty(\mathbb N)$,
	\[
	h_{\mathrm{top}}\bigl(f,\mathcal N_\vartheta(f,M)\bigr)\ge h_\mu(f)>0.
	\]
	If in addition $\dim(M)=2$ and $h_{\mathrm{top}}(f,M)>0$, then
	\[
	h_{\mathrm{top}}\bigl(f,\mathcal N_\vartheta(f,M)\bigr)
	= h_{\mathrm{top}}(f,M)>0.
	\]
\end{maincorollary}

Finally, we turn to $C^0$-generic topological dynamics. In contrast to the $C^1$ setting, where structural stability of Morse-Smale diffeomorphisms prevents Bohr chaotic maps from being always $C^1$-dense, the $C^0$ situation is markedly different.
Let $M$ be a compact smooth Riemannian manifold with dimension $\operatorname{dim}(M)\geq1$. Let $C^0(M)$ denote the space of all continuous self-maps of $M$, endowed with
the metric
\[
d_{C^0}(g_1,g_2):=\max_{x\in M} d(g_1(x),g_2(x)).
\]
Let $\mathcal H(M)$ denote the space of all homeomorphisms of $M$, endowed with
the metric
\[
d_{\mathcal H}(g_1,g_2):=\max\{d_{C^0}(g_1,g_2),
d_{C^0}(g_1^{-1},g_2^{-1})\}.
\]
For notational convenience, let $S(M)$ denote either $C^0(M)$ when
$\dim(M)\ge 1$ or $\mathcal H(M)$ when $\dim(M)\ge 2$. It is known that there exists a residual subset $\mathcal R\subset S(M)$ such
that every $f\in\mathcal R$ has infinite topological entropy \cite{Yano-1980} and satisfies the
shadowing property \cite{Ko\'scielniak-2005,Mazur-Oprocha-2013}. Therefore Theorem \ref{Theorem B} gives the following $C^0$-generic consequence.

\begin{maincorollary}
	Let $M$ be a compact smooth Riemannian manifold with dimension $\operatorname{dim}(M)\geq1$. Then there exists
	a residual subset $\mathcal R\subset S(M)$ such that for every
	$f\in \mathcal R$ and every non-trivial weight
	$\vartheta\in \ell^\infty(\mathbb N)$,
	\[
	h_{\mathrm{top}}\bigl(f,\mathcal N_\vartheta(f,M)\bigr)
	= h_{\mathrm{top}}(f,M)=\infty.
	\]
	In particular, every $f\in \mathcal R$ is Bohr chaotic.
\end{maincorollary}

\begin{remark}
	When $M$ is the circle $S^1$ or the interval $[0,1]$, for any $f\in\mathcal{H}(M)$, one has $h_{\mathrm{top}}(f,M)=0$, see \cite[Theorem 7.14 and Corollary 7.14.1]{Walters-1982}, hence $(M,f)$ is not Bohr chaotic.
\end{remark}

We conclude this subsection with a brief discussion of the main ideas. The proofs of Theorems \ref{Theorem B} and \ref{Theorem C} follow a common strategy: we construct semi-horseshoes whose iterate length is sufficiently large, whose different iterates are pairwise disjoint, and whose entropy is arbitrarily close to that of the ambient system (Theorems \ref{Theorem 4.1} and \ref{Theorem 5.1}). In the shadowing case, the starting point is a high-entropy separated set as in \cite{DOT}. In the specification setting, Tal's marker construction \cite{Tal2023} provides an important source of inspiration. The need to impose a sufficiently large iterate length already appears in our analysis of the simpler case of full shifts; see Lemma \ref{Lemma A}.

The main difficulty is to achieve three requirements simultaneously: large entropy, pairwise disjoint iterates, and sufficiently long iterate length. To overcome this, we first construct separated sets with better recurrence control (Lemma \ref{Lemma B}), and then use Furstenberg's measure-theoretic multiple recurrence theorem to obtain special marker segments adapted to the shadowing property and the modified almost specification property (Lemma \ref{Lemma C-New}). Once these ingredients are in place, the proofs are completed by a careful control of the constants arising in the construction. Our techniques do not apply to systems satisfying only the vague specification property. Indeed, Can and Trilles \cite{Can-Trilles-2025} constructed a minimal dynamical system with positive topological entropy satisfying the vague specification property for which no semi-horseshoe exists.

\textbf{Organization of this paper.}
Section \ref{Section 2} collects the preliminaries on transitivity, shadowing,
specification-like properties, topological entropy of subsets, separated sets and
symbolic dynamics. Section \ref{Section 3} is devoted to the proof of Theorem \ref{Theorem A}. Section \ref{Section 4} proves
several key lemmas for the proofs of Theorems \ref{Theorem B} and \ref{Theorem C}, including Lemmas \ref{Lemma A}, \ref{Lemma B}, and
\ref{Lemma C-New}. Sections \ref{Section 5} and \ref{Section 6} are devoted to
the proofs of Theorems \ref{Theorem B} and \ref{Theorem C}, respectively.
Section \ref{Section 7} is a detailed explanation for Remarks \ref{Remark 1.1} and
\ref{Remark 1.2}. Finally, Section \ref{Section 8} presents
applications of the main results, including the proof of Corollary \ref{Corollary B}.

\section{Preliminaries}\label{Section 2}
\subsection{Topologically transitive dynamical systems}
We say that a dynamical system $(X,f)$ is \emph{topologically transitive}, if for any nonempty open subsets $U,V\subset X$, there exists $n\in\mathbb{N}$ such that $U\cap f^{-n}(V)\neq\emptyset$. Given $x\in X$ and $\varepsilon>0$, we denote $B(x,\varepsilon):=\{y\in X: d(x,y)<\varepsilon\}$. For a topologically transitive dynamical system, we have the following fact.
\begin{lemma}\label{Lemma 2.1}
	Suppose that $(X,f)$ is a topologically transitive dynamical system. Then for any $\mu\in\mathcal{M}(X)$, any $0<\delta<1$, any $\varepsilon>0$ and any $m$ points $\{x_i\}_{i=1}^{m}$ in $X$, there exist $E\subset X$ with $\mu(E)>\delta$ and $K\in\mathbb{N}$ such that for any $x\in E$ and any $1\leq i\leq m$, there exist $l_{0i},l_{i0}\in\mathbb{N}$ with $$0\leq l_{0i},l_{i0}\leq K$$ such that $$B(x,\varepsilon)\cap f^{-l_{0i}}(B(x_i,\varepsilon))\neq\emptyset\text{ and }B(x_i,\varepsilon)\cap f^{-l_{i0}}(B(x,\varepsilon))\neq\emptyset.$$
\end{lemma}
\begin{proof}
	We fix $\mu\in\mathcal{M}(X)$, $0<\delta<1$ and $\varepsilon>0$. Since  $(X,f)$ is topologically transitive, for any $x\in X$ and any $1\leq i\leq m$, there exist $l_{0i}(x),l_{i0}(x)\in\mathbb{N}$ such that $$B(x,\varepsilon)\cap f^{-l_{0i}(x)}(B(x_i,\varepsilon))\neq\emptyset\text{ and }B(x_i,\varepsilon)\cap f^{-l_{i0}(x)}(B(x,\varepsilon))\neq\emptyset.$$ Then $$X=\bigcup_{k\geq0}^{\infty}\{x\in X:\max\{l_{0i}(x),l_{i0}(x)\}\leq k\}.$$ Since $\mu (X)=1$, we can find $K\in\mathbb{N}$ large enough such that $$\mu(\{x\in X:\max\{l_{0i}(x),l_{i0}(x)\}\leq K\})>\delta.$$ We define $$E:=\{x\in X:\max\{l_{0i}(x),l_{i0}(x)\}\leq K\}.$$ It can be checked that $E$ satisfies the requirements.
\end{proof}

\subsection{Shadowing property}
Fix $\varepsilon>0$ and $\delta>0$, a sequence $\{x_n\}_{n=0}^{\infty}$ in $X$ is said to be a \emph{$\delta$-pseudo orbit} if $d(f(x_i),x_{i+1})<\delta$ for all $i\in\mathbb{N}$. A point $x\in X$ is \emph{$\varepsilon$-tracing} a pseudo-orbit $\{x_n\}_{n=0}^{\infty}$ if $d(f^i(x),x_i)<\varepsilon$ for all $i\in\mathbb{N}$. 
\begin{definition}\label{Definition 2.2}
	A topological dynamical system $(X,f)$ is said to satisfy \emph{shadowing property} if for any $\varepsilon>0$, there is $\delta>0$ such that every $\delta$-pseudo orbit of $f$ is $\varepsilon$-traced by some point in $X$. 
\end{definition}

\subsection{Specification-like properties}\label{Section 2.3}
	Let us recall the definitions of the well-known (forward/almost) weak specification property, backward weak specification property and almost specification property. 
	\begin{definition}
		Given $\varepsilon_0>0$, a function $$g(n,\varepsilon):\mathbb{N}^+\times(0,\varepsilon_0]\rightarrow\mathbb{N}$$ is called a \emph{mistake function} if for any $m\in\mathbb{N}^+$ and $0<\varepsilon\leq\varepsilon_0$, we have $$g(m,\varepsilon)\leq g(m+1,\varepsilon)\text{ and }\lim_{n\to\infty}\frac{g(n,\varepsilon)}{n}=0.$$ For convenience, we set $g(n,\varepsilon)=g(n,\varepsilon_0)$ for any $\varepsilon>\varepsilon_0$.
	\end{definition}
	\begin{definition}
		We say that a dynamical system $(X,f)$ satisfies the \emph{(forward/almost) weak specification property} with \emph{mistake function} $g$ if for any $k\in\mathbb{N}^+$, any $x_1,\cdots,x_k\in X$ and any non-negative integers $a_1,\cdots, a_k$ and $b_1,\cdots,b_k$ with
		\begin{equation*}
			a_1\leq b_1<\cdots<a_k\leq b_k
		\end{equation*}
		and
		\begin{equation}\label{equ (2.1)}
			a_{i+1}-b_i\geq g(b_{i+1}-a_{i+1}+1,\varepsilon)
		\end{equation}
		for any $1\leq i\leq k-1$, there exists a point $z\in X$ such that
		\begin{equation*}
			d(f^nz,f^{n}x_i)<\varepsilon,\quad\text{for any}\quad a_i\leq n\leq b_i\text{ and any }1\leq i\leq k.
		\end{equation*}
		 If the equation (\ref{equ (2.1)}) is replaced by $$a_{i+1}-b_i\geq g(b_i-a_i+1,\varepsilon),$$ then we say $(X,f)$ satisfies the \emph{backward weak specification property} with mistake function $g$.
	\end{definition}
	Denote $\Lambda_n:=\{0,1,\cdots,n-1\}$ and $d_\Lambda(f;x,y):=\max\{d(f^ix,f^iy):i\in\Lambda\}$ for any non-empty subset $\Lambda\subset\Lambda_n$. For any set $A$, let $|A|$ denote its cardinality.
	\begin{definition}
		Let $g$ be a mistake function, $\varepsilon>0$ and $n\in\mathbb{N}^+$. The \emph{mistake dynamical ball} $B_n(f;g(n,\varepsilon);x,\varepsilon)$ of radius $\varepsilon$ and length $n$ associated to $g$ for dynamical system $(X,f)$ is defined by
		$$B_n(f;g(n,\varepsilon);x,\varepsilon):=\{y\in X:d_\Lambda(f;x,y)<\varepsilon\text{ for some }\Lambda\in I(f;g(n,\varepsilon);n,\varepsilon)\},$$where $I(f;g(n,\varepsilon);n,\varepsilon):=\{\emptyset\neq\Lambda\subset\Lambda_n:|\Lambda|\geq n-g(n,\varepsilon)\}.$
	\end{definition}
	\begin{definition}\label{Definition almost spec}
		A topological dynamical system $(X,f)$ is said to satisfy \emph{almost specification property} with mistake function $g$, if there exists a function $k_g:(0,\infty)\to\mathbb{N}^+$ such that for any $m\in\mathbb{N}^+$, any $\varepsilon_1,\varepsilon_2\cdots,\varepsilon_m>0$, any points $x_1,x_2,\cdots,x_m\in X$, and any integers $n_1\geq k_g(\varepsilon_1),\cdots,n_m\geq k_g(\varepsilon_m)$, we have $$\bigcap_{j=1}^mf^{-l_j}(B_{n_j}(f;g(n_j,\varepsilon_j);x_j,\varepsilon_j))\neq\emptyset,$$ where $n_0=0$ and $l_j=\sum_{s=0}^{j-1}n_s$ for any $1\leq j\leq m$.
	\end{definition}
	We refer to \cite[Section 3]{Kwietniak-Lacka-Oprocha-2016} for a detailed introduction to the above specification-like properties. In order to make our main results applicable to dynamical systems that satisfy weak specification property, backward weak specification property or almost specification property, we modify Definition \ref{Definition almost spec} to define the modified almost specification property as follows:
	\begin{definition}\label{Definition 2.3}
		A topological dynamical system $(X,f)$ is said to satisfy the \emph{ modified almost specification property} with mistake function $g$, if there exists a function $k_g:(0,\infty)\to\mathbb{N}^+$ such that for any $m\in\mathbb{N}^+$, any $\varepsilon>0$, any points $x_1,x_2,\cdots,x_m\in \bigcap_{i=0}^\infty f^i(X)$ and any integer $n\geq k_g(\varepsilon)$, we have $$\bigcap_{j=1}^mf^{-(j-1)n}(B_{n}(f;g(n,\varepsilon);x_j,\varepsilon))\neq\emptyset.$$ 
	\end{definition} 
It is clear that if a dynamical system satisfies the almost specification property, then it satisfies the modified almost specification property. Next, we will show that the modified almost specification is also weaker than the weak specification property and backward weak specification property.
\begin{lemma}\label{Lemma 5.2}
	Suppose that $(X,f)$ is a dynamical system satisfying the weak specification property (resp. backward weak specification property) with the mistake function $g$. Then for any $k\in\mathbb{N}^+$, any $x_1,\cdots,x_k\in\bigcap_{i=0}^\infty f^i(X)$ and any non-negative integers $a_1,\cdots, a_k$ and $b_1,\cdots,b_k$ with
	\begin{equation*}
		a_1\leq b_1<\cdots<a_k\leq b_k
	\end{equation*}
	and
	\begin{equation*}
		a_{i+1}-b_i\geq g(b_{i+1}-a_{i+1}+1,\varepsilon)~(\text{resp. }a_{i+1}-b_i\geq g(b_i-a_i+1,\varepsilon))
	\end{equation*}
	for any $1\leq i\leq k-1$, there exists a point $z\in X$ such that
	\begin{equation*}
		d(f^nz,f^{n-a_i}x_i)\leq\varepsilon,\quad\text{for any}\quad a_i\leq n\leq b_i\text{ and any }1\leq i\leq k.
	\end{equation*}
\end{lemma}
\begin{proof}
	Since $x_1,\cdots,x_k\in \bigcap_{i=0}^\infty f^i(X)$, we can choose $y_1,\cdots,y_k\in X$ such that $f^{a_i}y_i=x_i$ for any $1\leq i\leq k$. By the weak specification property (resp. backward weak specification property), there exists a point $z\in X$ such that
	\begin{equation*}
		d(f^nz,f^{n}y_i)\leq\varepsilon,\quad\text{for any}\quad a_i\leq n\leq b_i\text{ and any }1\leq i\leq k.
	\end{equation*}
	Hence, 
	\begin{equation*}
		d(f^nz,f^{n-a_i}x_i)\leq\varepsilon,\quad\text{for any}\quad a_i\leq n\leq b_i\text{ and any }1\leq i\leq k.
	\end{equation*}
\end{proof}
\begin{lemma}\label{Lemma 2.10-New}
	Suppose that $(X,f)$ is a dynamical system satisfying the weak specification property (resp. backward weak specification property) with the mistake function $g$. Then it satisfies the modified almost specification property with the mistake function $g$. 
\end{lemma}
\begin{proof}
	Given $\varepsilon>0$, denote $$k_g(\varepsilon):=\min\{m\in\mathbb{N}^+:g(n,\varepsilon)\leq n-1\text{ for any }n\geq m\}.$$ For any $m\in\mathbb{N}^+$, any $x_1,x_2,\cdots,x_m\in \bigcap_{i=0}^\infty f^i(X)$ and any integer $n\geq k_g(\varepsilon)$, we denote $a_i=(i-1)n$ and $b_i=a_i+n-1-g(n,\varepsilon)$ for any $1\leq i\leq m$. Then for any $1\leq i\leq m-1$, we have $$a_{i+1}-b_i=g(n,\varepsilon)+1\geq g(n-g(n,\varepsilon),\varepsilon)=g(b_i-a_i+1,\varepsilon)=g(b_{i+1}-a_{i+1}+1,\varepsilon).$$ Hence, by Lemma \ref{Lemma 5.2}, we have $$\bigcap_{j=1}^mf^{-(j-1)n}(B_{n}(f;g(n,\varepsilon);x_j,\varepsilon))\neq\emptyset.$$ 
\end{proof}

\subsection{Bowen topological entropies of subsets}\label{Section Bowen entropy}
For $x,  y\in X,$ $\varepsilon>0$ and $n\in\mathbb{N}^+$,   define
$$d_n(f;x,  y):=\max\{d(f^i(x),  f^i(y)):i=0,  1,  \cdots,  n-1\}$$
and 
$$B_n(f;x,  \varepsilon):=\{z\in X:d_n(f;x,  z)<\varepsilon\}.  $$
\begin{definition}\label{Definition 2.4}
	Let $E\subseteq X$, $\sigma>0$ and $\mathcal {G}_{n}(E,  \sigma)$ be the collection of all finite or countable covers of $E$ by sets of the form $B_{u}(f;x,  \sigma)$ with $u\geq n$. Let $t\geq0$, we set
	$$C(E;t,  n,  \sigma,  f):=\inf_{\mathcal {C}\in \mathcal {G}_{n}(E,  \sigma)}\sum_{B_{u}(f;x,  \sigma)\in \mathcal {C}}e^{-tu}$$ and $$
	C(E;t,  \sigma,  f):=\lim_{n\rightarrow\infty}C(E;t,  n,  \sigma,  f).  $$
	Then we define
	$$h_{\mathrm{top}}(E;\sigma,  f):=\inf\{t:C(E;t,  \sigma,  f)=0\}=\sup\{t:C(E;t,  \sigma,  f)=\infty\}$$
	and 
	$$ h_{\mathrm{top}}(f,  E):=\lim_{\sigma\rightarrow0} h_{\mathrm{top}}(E;\sigma,  f)$$
	exists and is called \emph{Bowen topological entropy} of $E.$ 
\end{definition}
In particular, when $(E,f)$ is a dynamical system, $h_{\mathrm{top}}(f,  E)$ is equal to its classical topological entropy.

The following are some basic properties.
\begin{lemma}\label{Lemma 2.2}\cite[Proposition 2]{Bowen1973}
	Suppose that $(X,f)$ is a dynamical system and $Y, Y_1, Y_2,\cdots\subset X$. Then we have
	\begin{enumerate}[(1)]
		\item $h_{\mathrm{top}}(f,f(Y))=h_{\mathrm{top}}(f,Y)$;
		\item $h_{\mathrm{top}}(f,\bigcup_{i\geq1}Y_i)=\sup_{i\geq1}h_{\mathrm{top}}(Y_i)$;
		\item $h_{\mathrm{top}}(f^n,Y)=nh_{\mathrm{top}}(f,Y)$ for any $n\geq1$.
	\end{enumerate}
\end{lemma}
A surjective continuous map $\pi:(X,f)\to(Y,g)$ between two dynamical systems is called as a \emph{factor} if $g\circ\pi=\pi\circ f$.
\begin{lemma}\label{Lemma 2.3}\cite[Theorem 3.11]{Feng-Huang-2012}
	Suppose that $\pi:(X,f)\to(Y,g)$ is a factor between two dynamical systems. Then for any $E\subset X$, one has $$h_{\mathrm{top}}(f,E)\geq h_{\mathrm{top}}(g,\pi(E)).$$
\end{lemma}

\subsection{Separated sets}
For $\varepsilon>0$ and $n\in\mathbb{N}^+$,   two points $x$ and $y$ are $(f;n,  \varepsilon)$-separated if
$d_n(f;x,y)>\varepsilon$.
A subset $E$ is $(f;n,  \varepsilon)$-separated if any two different points of $E$ are $(f;n,  \varepsilon)$-separated. Now, for $\delta >0$, $\varepsilon >0$ and $n\in\mathbb{N}^+$, two points $x$ and $y$ are $(f;\delta,n,\varepsilon)$-separated if
\begin{equation*}
	|\{0\leq j\leq n-1:d(f^{j}x,f^{j}y)>\varepsilon\}| \geq \delta n.
\end{equation*}
A subset $E$ is $(f;\delta,n,\varepsilon)$-separated  if any two different points of $E$ are $(f;\delta,n,\varepsilon)$-separated. It is clear that every $(f;\delta,n,\varepsilon)$-separated set is $(f;n,\varepsilon)$-separated.  
For any $\mu\in\mathcal{M}_f(X)$, let $h_\mu(f)$ denote the metric entropy of $\mu$ and $S_\mu$ denote its support, that is $$S_\mu:=\{x\in X:\mu(U)>0\text{ for any neighborhood }U\text{ of }x\}.$$
\begin{lemma}\cite[Proposition 2.1]{Pfister-Sullivan2005}\label{Lemma 2.4}
	Suppose that $\mu\in\mathcal{M}_f^e(X)$ with $h_\mu(f)>0$ and $\eta<h_\mu(f)$. Then there exist $\delta>0$, $\varepsilon>0$ and $N\in\mathbb{N}^+$ such that for any $n\geq N$,   there exists a $(f;\delta, n,  \varepsilon)$-separated set {$\Gamma_n\subset  S_\mu$} with
	$$|\Gamma_n|\geq e^{n\eta}.  $$
\end{lemma}

\subsection{Symbolic dynamics}
Let $\mathcal{A}$ be a finite alphabet with $|\mathcal{A}|\geq2$ and $$\mathcal{A}^{\mathbb{N}}:=\{x_0x_1x_2\cdots:x_i\in\mathcal{A}\text{ for any }i\in\mathbb{N}\}$$ with the product topology induced by the discrete topology on $\mathcal{A}$. The metric on $\mathcal{A}^{\mathbb{N}}$ is defined as $$d(x,y):=e^{-\min\{i\in\mathbb{N}:~x_i\neq y_i\}}$$ for any $x\neq y\in\mathcal{A}^{\mathbb{N}}$ with $x=x_0x_1x_2\cdots$ and $y=y_0y_1y_2\cdots$. The one-sided full shift over $\mathcal{A}$ is defined by $(\mathcal{A}^{\mathbb{N}},\sigma)$, where
\begin{equation*}
	\sigma(x)_i=x_{i+1}\text{ for any } x\in\mathcal{A}^{\mathbb{N}}\text{ and }i\in\mathbb{N}.
\end{equation*}
A subset $C\subset\mathcal{A}^{\mathbb{N}}$ is said to be a \emph{cylinder} if there exists $n\in\mathbb{N}^+$ and $a_0,a_1,\cdots,a_{n-1}\in\mathcal{A}$ such that $$C=[a_0,a_1,\cdots,a_{n-1}]:=\{x\in\mathcal{A}^{\mathbb{N}}:x_i=a_i\text{ for any }0\leq i\leq n-1\}.$$

Similarly, we define $$\mathcal{A}^{\mathbb{Z}}:=\{\cdots x_{-1}x_0x_1x_2\cdots:x_i\in\mathcal{A}\text{ for any }i\in\mathbb{Z}\}$$ with the product topology induced by the discrete topology on $\mathcal{A}$. The metric on $\mathcal{A}^{\mathbb{Z}}$ is defined as $$d(x,y):=e^{-\min\{|i|\in\mathbb{Z}:~x_i\neq y_i\}}$$ for any $x\neq y\in\mathcal{A}^{\mathbb{Z}}$ with $x=\cdots x_{-1}x_0x_1x_2\cdots$ and $y=\cdots y_{-1}y_0y_1y_2\cdots$. The two-sided full shift over $\mathcal{A}$ is defined by $(\mathcal{A}^{\mathbb{Z}},\sigma)$, where
\begin{equation*}
	\sigma(x)_i=x_{i+1}\text{ for any } x\in\mathcal{A}^{\mathbb{Z}}\text{ and }i\in\mathbb{Z}.
\end{equation*}
It is well known that we always have $h_{\mathrm{top}}(\sigma,\mathcal{A}^{\mathbb{Z}})=h_{\mathrm{top}}(\sigma,\mathcal{A}^{\mathbb{N}})=\log |\mathcal{A}|$.

A member $w$ of $\mathcal{A}^{\{i,i+1,\cdots,j\}}$ for some integers $i\leq j$ is called a \emph{word} over $\mathcal{A}$. Let $\mathcal{A}^\ast:=\bigcup_{i\leq j}\mathcal{A}^{\{i,i+1,\cdots,j\}}$ denote the set of all words over $\mathcal{A}$. Given $w\in\mathcal{A}^\ast$ and $n\in\mathbb{N}^{+}\cup\{\infty\}$, we define $w^n:=\underbrace{ww\cdots w}_\text{$n$ items}$.

\section{Proof of Theorem \ref{Theorem A}}\label{Section 3}
In this section, we give the proof of Theorem \ref{Theorem A}. First, we recall some basic facts obtained in \cite{Fan-Fan-Ryzhikov-Shen2022} for Bohr chaoticity. 
\begin{lemma}\cite[Proposition 3.1]{Fan-Fan-Ryzhikov-Shen2022}\label{new-lemma3.1}
	Suppose that $\pi:(X,f)\to(Y,g)$ is a factor between two dynamical systems and $(Y,g)$ is Bohr chaotic. Then $(X,f)$ is Bohr chaotic.
\end{lemma}
Let $L\in\mathbb{N}^+$ and $\Lambda\subset X$ is compact and $f^L$-invariant, i.e., $f^L(\Lambda)\subset\Lambda$. we say that $L$ is the \emph{first return time} of $\Lambda$ if $$\Lambda\cap f^i(\Lambda)=\emptyset\text{ for any }1\leq i\leq L-1.$$ 
\begin{lemma}\cite[Theorem 3.3]{Fan-Fan-Ryzhikov-Shen2022}\label{new-lemma3.2}
	Let $(X,f)$ be a dynamical system, and let $L\in\mathbb{N}^+$. If a subsystem $(Y,f^L)$ of $(X,f^L)$ is Bohr chaotic and $L$ is the first return time of $Y$, then $(X,f)$ is Bohr chaotic.
\end{lemma}
\begin{lemma}\cite[Proposition 8.4]{Fan-Fan-Ryzhikov-Shen2022}\label{new-lemma3.3}
	For any cylinder $C\subset\{0,1\}^{\mathbb{N}}$, there exists $N\in\mathbb{N}^+$ and a one-sided $N$-order horseshoe $(\Lambda,\sigma^N)$ of $(\{0,1\}^{\mathbb{N}},\sigma)$ such that $\Lambda\subset C$ and $$\Lambda\cap\sigma^i(\Lambda)=\emptyset\text{ for any }1\leq i\leq N-1.$$ 
\end{lemma}
The following lemma shows that it is enough to consider the one-sided case.
\begin{lemma}\label{new-lemma3.4}
	Suppose that $(X,f)$ has a two-sided $N$-order semi-horseshoe, then it has a one-sided $N$-order semi-horseshoe.
\end{lemma}
\begin{proof}
	Suppose that $(X,f)$ has a two-sided $N$-order semi-horseshoe $(\Lambda,f^N)$ and $\pi:(\Lambda,f^N)\to(\{0,1\}^{\mathbb{Z}},\sigma)$ is a factor. We define $g:(\{0,1\}^{\mathbb{Z}},\sigma)\to(\{0,1\}^{\mathbb{N}},\sigma)$ as $$g(x)_i=x_i\text{ for any }x\in\{0,1\}^{\mathbb{Z}}\text{ and }i\in\mathbb{N}.$$Then it can be checked that $g$ is a factor and thus $g\circ\pi:(\Lambda,f^N)\to(\{0,1\}^{\mathbb{N}},\sigma)$ is a factor. As a result, $(X,f)$ has a one-sided $N$-order semi-horseshoe.
\end{proof}
Next, we show how to construct a one-sided $M$-order horseshoe of $(\{0,1\}^{\mathbb{N}},\sigma)$ with disjoint steps for any $M\neq2$.
\begin{lemma}\label{new-lemma3.5}
	Given an integer $M=1$ or $M\geq 3$, $(\{0,1\}^{\mathbb{N}},\sigma)$ has a one-sided $M$-order horseshoe $(\Xi_M,\sigma^M)$ such that $$\Xi_M\cap\sigma^i(\Xi_M)=\emptyset\text{ for any }1\leq i\leq M-1.$$ 
\end{lemma}
\begin{proof}
	When $M=1$, it is clear that we can define $\Xi_M:=\{0,1\}^{\mathbb{N}}$. Now, we assume that $M\geq3$. We define $\omega_1:=0^{M-1}1$, $\omega_2:=0^{M-2}11$ and an alphabet $$\mathcal{A}:=\{\omega_1,\omega_2\},$$ then we can define a compact subset $\Xi_M$ of $\{0,1\}^\mathbb{N}$ as $$\Xi_M:=\{a=a_0a_1\cdots:a_i\in\mathcal{A}\text{ for any }i\in\mathbb{N}\}.$$ As a result, we have that $\sigma^M(a)_i=a_{i+1}$ for any $a\in\Xi_M$ and $i\in\mathbb{N}$. It is clear that $$(\Xi_M,\sigma^M)\text{ is topologically conjugate to }  (\{0,1\}^\mathbb{N},\sigma).$$ It remains to prove the disjointness.
	
	Now suppose, towards a contradiction, that there exist $1\le j\le M-1$ and
	\[
	x\in \Xi_M\cap \sigma^j(\Xi_M).
	\]
	Write $x=\sigma^jy$ with $y\in \Xi_M$. We always have that $x_i=y_i=0$ for any $0\leq i\leq M-2$, $x_{M}=y_{M}=1$ and $x_{M+1}=y_{M+1}=0$
	
	If $j=1$, then by $x_{M}=1$ and $y_{M+1}=0$, we have that $x_{M}\neq y_{M+1}$, it is a contradiction. 
	
	If $2\leq j\leq M-2$, then $2\leq M-j\leq M-2$ and by $y_{M}=1$, we have that $x_{M-j}=1$, it is a contradiction to $x_i=0$ for any $0\leq i\leq M-2$. 
	
	If $j=M-1$, then by $y_{M}=1$ we have $x_1=1$. Since $M\geq3$, we have $x_2=1$, it is a contradiction to $y_{M+1}=0$.
	
	Therefore, 
	$$\Xi_M\cap\sigma^i(\Xi_M)=\emptyset\text{ for any }1\leq i\leq M-1.$$ 
\end{proof}
\begin{remark}
	Lemma \ref{new-lemma3.5} is not hold for $M=2$. In fact, if $(\{0,1\}^{\mathbb{N}},\sigma)$ has a one-sided $2$-order horseshoe $(\Xi_2,\sigma^2)$ such that $\Xi_2\cap\sigma(\Xi_2)=\emptyset$, then $(\Xi_2,\sigma^2)$ has two fixed points and $0^\infty,1^\infty\notin\Xi_2$. As a result, $(01)^\infty,(10)^\infty\in\Xi_2$ and thus $\Xi_2\cap\sigma(\Xi_2)\neq\emptyset$, it is a contradiction.
\end{remark}
\subsection{The proof of Theorem \ref{Theorem A}} By Lemma \ref{new-lemma3.4}, we can always assume that $(X,f)$ has a one-sided $N$-order semi-horseshoe $(\Lambda,f^N)$ for some $N\in\mathbb{N}^+$ and $\pi:(\Lambda,f^N)\to(\{0,1\}^{\mathbb{N}},\sigma)$ is the factor. Choose an positive integer $M$ as follows. 
\begin{equation}\label{equ-3.1}
	M:=\begin{cases}
		N & \text{if }N\neq 2,\\
		4 & \text{if }N=2.
	\end{cases}
\end{equation}
Then $N$ divides $M$, i.e., $N\mid M$. We set $\tau:=NM$ and let $\Xi_M$ be the set given by Lemma \ref{new-lemma3.5} and $\theta_M:(\Xi_M,\sigma^M)\to(\{0,1\}^{\mathbb{N}},\sigma)$ be the conjugacy. We define  $$B_0:=\pi^{-1}(\Xi_M)\subset\Lambda\text{ and }\rho:=\theta_M\circ\pi:B_0\to\{0,1\}^{\mathbb{N}}.$$ Then $$\rho(B_0)=\{0,1\}^{\mathbb{N}}\text{ and }\rho\circ f^\tau=\theta_M\circ\pi\circ f^{\tau}=\theta_M\circ\sigma^M\circ\pi=\sigma\circ\theta_M\circ\pi=\sigma\circ\rho.$$ For any $x\in B_0$, we have $\pi(f^\tau x)=\sigma^M(\pi (x))\in \Xi_M$ and thus $f^\tau x\in B_0$. Hence, $f^\tau(B_0)\subset B_0$. We claim that $$B_0\cap f^{Nl}(B_0)=\emptyset\text{ for any }1\leq l\leq M-1.$$ In fact, if there exists $x\in B_0\cap f^{Nl}(B_0)$ for some $1\leq l\leq M-1$, say $x=f^{Nl}y$ with $y\in B_0$, then $$\pi(x)=\pi(f^{Nl}y)=\sigma^l(\pi (y))\in \Xi_M\cap \sigma^l(\Xi_M),$$contrary to Lemma \ref{new-lemma3.5}.
\begin{definition}
	A compact set $B\subset B_0$ is called \emph{admissible} if
	\begin{enumerate}
		\item [(A1)] $f^{\tau}(B)\subset B$;
		\item [(A2)] $\rho(B)=\{0,1\}^{\mathbb{N}}$.
	\end{enumerate}
\end{definition}

By the discussion above, $B_0$ is admissible. Denote $$Q:=\{1,2,\cdots,\tau-1\}\setminus \{Nl:l\in\mathbb{N}^+\}.$$ For an admissible set $B$ and $q\in Q$, define $$Y_q(B):=\rho(B\cap f^q(B))\subset\{0,1\}^{\mathbb{N}}.$$
\begin{lemma}\label{new-lemma3.8}
	Let $B\subset B_0$ be admissible. Assume that $Y_q(B)\subsetneq \{0,1\}^{\mathbb{N}}$ for any $q\in Q$, then there exists $L\in\mathbb{N}^+$ such that $(X,f)$ has a one-sided $\tau L$-order semi-horseshoe $(E,f^{\tau L})$ such that $E\subset B$ and $$E\cap f^i(E)=\emptyset\text{ for any }1\leq i\leq \tau L-1.$$ In particular, $(X,f)$ is Bohr chaotic.
\end{lemma}
\begin{proof}
	Given $q\in Q$, since $B\cap f^q(B)$ is compact and $\rho$ is continuous, we have that $Y_q(B)=\rho(B\cap f^q(B))\subset\{0,1\}^{\mathbb{N}}$ is compact. For any $x\in Y_q(B)$, say $x=\rho(y)$ with $y\in B\cap f^q(B)$. Write $y=f^q(z)$ with $z\in B$. Then by $f^\tau(B)\subset B$, $$f^\tau y= f^\tau(f^qz)=f^q(f^\tau z)\in B\cap f^q(B)$$ and $$\sigma(x)=\sigma(\rho(y))=\rho (f^\tau y)\in \rho(B\cap f^q(B))=Y_q(B).$$ Hence, $\sigma(Y_q(B))\subset Y_q(B)$.
	
	We claim that $Y_q(B)$ has empty interior. In fact, if $Y_q(B)$ contains a cylinder $C=[a_0\dots a_{m-1}]$ of $\{0,1\}^{\mathbb{N}}$, then $$\{0,1\}^{\mathbb{N}}=\sigma^m(C)\subset\sigma^m(Y_q(B))\subset Y_q(B),$$ contrary to $Y_q(B)\subsetneq \{0,1\}^{\mathbb{N}}$.
	
	We define $$Y:=\bigcup_{q\in Q}Y_q(B).$$ Since, $Q$ is finite and $\{0,1\}^{\mathbb{N}}$ is a Baire space, we have $Y$ is closed and $\{0,1\}^{\mathbb{N}}\setminus Y\neq\emptyset$. We choose a cylinder $C\subset\{0,1\}^{\mathbb{N}}\setminus Y$. By Lemma \ref{new-lemma3.3}, there exist $L\in\mathbb{N}^+$ and a one-sided $L$-order horseshoe $(\Theta,\sigma^L)$ of $(\{0,1\}^{\mathbb{N}},\sigma)$ such that $\Theta\subset C$ and $$\Theta\cap\sigma^i(\Theta)=\emptyset\text{ for any }1\leq i\leq L-1.$$ 
	Define
	\[
	E:=(\rho|_{B})^{-1}(\Theta)\subset B.
	\]
	Then $E$ is compact and $f^{\tau L}(E)\subset E$. Indeed, if $x\in E$, then $f^{\tau L}x\in B$ and 
	\[
	\rho(f^{\tau L}x)=\sigma^L(\rho(x))\in \sigma^L(\Theta)\subset \Theta,
	\]
	so $f^{\tau L}x\in E$. 
	
	We claim that $\tau L$ is the first return time of $E$. Otherwise, there exists $x\in E\cap f^n(E)$ for some $1\leq n\leq \tau L-1$. Suppose that $x=f^n y$ with $y\in E$ and  
	\[
	n=a\tau+r,\text{ where } 0\leq a\leq L-1\text{ and }0\leq r\leq\tau-1.
	\]
	
	If $r=0$, then $1\leq a\leq L-1$, and
	\[
	\rho(x)=\rho(f^{a\tau}y)=\sigma^a(\rho(y))\in \Theta\cap \sigma^a(\Theta),
	\]
	contrary to the disjointness of the steps of $\Theta$.
	
	If $r=Nl$ for some $1\leq l\leq M-1$, let $z=f^{a\tau}y\in B$. Then
	\[
	x=f^{Nl}z\in B\cap f^{Nl}(B)\subset B_0\cap f^{Nl}(B_0),
	\]
	which contradicts the fact that $B_0\cap f^{Nl}(B_0)=\emptyset$ for any $1\leq l\leq M-1$.
	
	If $r\in Q$, let again $z=f^{a\tau}y\in B$. Then
	\[
	x=f^r z\in B\cap f^r(B),
	\]
	so
	\[
	\rho(x)\in Y_r(B).
	\]
	On the other hand, $x\in E$ implies $\rho(x)\in \Theta\subset C$, contradicting $C\cap Y_r(B)=\emptyset$.
	
	Therefore, $$E\cap f^i(E)=\emptyset\text{ for any }1\leq i\leq \tau L-1.$$ Since $(\Theta,\sigma^L)$ is topologically conjugate to $(\{0,1\}^{\mathbb{N}},\sigma)$, by \cite[Theorem 1.1]{Fan-Fan-Ryzhikov-Shen2022}, it is Bohr chaotic. By Lemma \ref{new-lemma3.1} and the fact that $\rho:(E,f^{\tau L})\to(\Theta,\sigma^L)$ is a factor, we have that $(E,f^{\tau L})$ is Bohr chaotic. Finally, by Lemma \ref{new-lemma3.2}, we conclude that $(X,f)$ is Bohr chaotic.
\end{proof}
To prove Theorem \ref{Theorem A}, it is enough to find an admissible set $B$ such that $Y_q(B)\subsetneq \{0,1\}^{\mathbb{N}}$ for any $q\in Q$. If some $Y_q(B)$ equals the full shift, we can cut $B$ by the corresponding overlap time.

\begin{lemma}\label{new-lemma3.9}
	Let $B\subset B_0$ be admissible and let $q\in Q$. If $Y_q(B)=\{0,1\}^{\mathbb{N}}$, then
	\[
	B^+:=B\cap f^q(B)
	\]
	is again admissible.
\end{lemma}

\begin{proof}
	The set $B^+$ is compact as the intersection of two compact sets.
	
	Since $\rho(B^+)=Y_q(B)=\{0,1\}^{\mathbb{N}}$, condition (A2) holds for $B^+$. If $x\in B^+$, say $x=f^q y$ with $y\in B$, then
	\[
	f^{\tau}x=f^{\tau}(f^qy)=f^q(f^{\tau}y)\in B\cap f^q(B)=B^+,
	\]
	so $f^{\tau}(B^+)\subset B^+$ and thus condition (A1) holds. 
	
	Therefore, $B^+$ is admissible.
\end{proof}
Starting from $B_0$, we may therefore perform the following pruning algorithm:
\begin{enumerate}[(1)]
	\item if $Y_q(B)\subsetneq \{0,1\}^{\mathbb{N}}$ for any $q\in Q$, stop;
	\item otherwise choose $q\in Q$ with $Y_q(B)=\{0,1\}^{\mathbb{N}}$ and replace $B$ by $B\cap f^q(B)$.
\end{enumerate}
To prove that this algorithm terminates, we need two lemmas.
\begin{lemma}\label{new-lemma3.10}
	Let $q_1,\cdots,q_s\in Q$, and define inductively compact sets
	\[
	B^{(0)}:=B_0,
	B^{(j+1)}:=B^{(j)}\cap f^{q_{j+1}}(B^{(j)})\text{ for any }0\leq j\leq s-1.
	\]
	Then
	\[
	B^{(s)}\subset \bigcap_{\varepsilon\in\{0,1\}^s} f^{\varepsilon_1 q_1+\cdots+\varepsilon_s q_s}(B_0).
	\]
	In particular, if a fixed $q\in Q$ occurs exactly $t$ times among $q_1,\cdots,q_s$, then
	\[
	B^{(s)}\subset B_0\cap f^q(B_0)\cap f^{2q}(B_0)\cap\cdots\cap f^{tq}(B_0).
	\]
\end{lemma}

\begin{proof}
	We argue by induction on $s$. The case $s=1$ is immediate from the definition of $B^{(1)}$.
	Assume the statement proved for $s$. Let
	\[
	S_s:=\left\{\varepsilon_1 q_1+\cdots+\varepsilon_s q_s:\ \varepsilon\in\{0,1\}^s\right\}.
	\]
	Then the induction hypothesis gives
	\[
	B^{(s)}\subset \bigcap_{t\in S_s} f^t(B_0).
	\]
	Hence
	\[
	f^{q_{s+1}}\bigl(B^{(s)}\bigr)\subset \bigcap_{t\in S_s} f^{q_{s+1}+t}(B_0).
	\]
	Since
	\[
	B^{(s+1)}=B^{(s)}\cap f^{q_{s+1}}(B^{(s)}),
	\]
	we obtain
	\[
	B^{(s+1)}\subset (\bigcap_{t\in S_s} f^t(B_0))\cap
	(\bigcap_{t\in S_s} f^{q_{s+1}+t}(B_0)),
	\]
	which is exactly the required inclusion for $s+1$.
	
	The final assertion follows by choosing subset sums that use only the occurrences of the fixed value $q$.
\end{proof}
\begin{lemma}\label{new-lemma3.11}
	Fix $q\in Q$, and write
	\[
	q=Na+r,\text{ where }0\le a\le M-1\text{ and }1\le r\le N-1.
	\]
	Let $\gcd(N,r)$ denote the greatest common divisor of $N$ and $r$. Set
	\[
	d:=\gcd(N,r)\text{ and }m_q:=\frac{N}{d}.
	\]
	Then, during the pruning algorithm, the value $q$ can be chosen at most $m_q-1$ times.
\end{lemma}

\begin{proof}
	Assume that $q$ is chosen $m_q$ times. Let $B$ be the admissible set obtained after these choices (and possibly other intermediate cuts). By Lemma~\ref{new-lemma3.10},
	\[
	B\subset B_0\cap f^{m_q q}(B_0).
	\]
	Now
	\[
	m_q q=\frac{N}{d}(Na+r)=N\left(\frac{N}{d}a+\frac{r}{d}\right)=Nk,
	\]
	where
	\[
	k:=\frac{N}{d}a+\frac{r}{d}\in\mathbb{N}^+.
	\]
	We claim that $M\nmid k$. Indeed, if $k=Mh$ for some integer $h\ge 1$, then
	\[
	q=Na+r=dk=dMh.
	\]
	Since $N\mid M$, this would imply $N\mid q$, contradicting $q\in Q$.
	
	Therefore $k\equiv \bar k\pmod M$ for some $1\leq \bar k\leq M-1$. Write $k=sM+\bar k$ with $s\ge 0$. Since $f^{\tau}(B_0)\subset B_0$ and $\tau=NM$, we have
	\[
	f^{Nk}(B_0)=f^{N\bar k}\bigl(f^{s\tau}(B_0)\bigr)\subset f^{N\bar k}(B_0).
	\]
	Hence
	\[
	B\subset B_0\cap f^{Nk}(B_0)\subset B_0\cap f^{N\bar k}(B_0)=\emptyset,
	\]
	contradicting the fact that $\rho(B)=\{0,1\}^{\mathbb{N}}$. This proves that $q$ can be chosen at most $m_q-1$ times.
\end{proof}
Now, we can prove Theorem \ref{Theorem A}.
\begin{proof}[Proof of Theorem \ref{Theorem A}]
	Start with the admissible set $B_0$. Apply the pruning algorithm described above. By Lemma \ref{new-lemma3.9}, every stage of the algorithm produces an admissible set.
	
	The set $Q$ is finite. By Lemma \ref{new-lemma3.11}, each $q\in Q$ can be chosen at most $m_q-1$ times. Therefore the pruning algorithm must terminate after finitely many steps, in fact after at most $\sum_{q\in Q}(m_q-1)$ cuts, yielding an admissible set $B_\ast\subset B_0$ such that
	\[
	Y_q(B_\ast)\subsetneq \{0,1\}^{\mathbb{N}}\text{ for any }q\in Q.
	\]
	By using Lemma \ref{new-lemma3.8} for $B_\ast$, we conclude that $(X,f)$ is Bohr chaotic.
\end{proof}

\section{Key Lemmas}\label{Section 4}
Before we give the proof of Theorem \ref{Theorem B}, in this section, we prove some key lemmas.
\subsection{Bohr chaoticity in one-sided full shifts} We consider the topological entropy of the set $\mathcal{N}_\vartheta$ in one-sided full shifts.
\begin{mainlemma}\label{Lemma A}
	Let $(\mathcal{A}^{\mathbb{N}},\sigma)$ be the one-sided full shift with $|\mathcal{A}|=2m$ for some $m\geq2$. Then for any non-trivial weight $\vartheta\in \ell^\infty(\mathbb{N})$, we have  $$h_{\mathrm{top}}(\sigma,\mathcal{N}_\vartheta(\sigma,\mathcal{A}^{\mathbb{N}}))\geq\log m.$$
\end{mainlemma}
\begin{proof}
	Without loss of generality, we assume that $\mathcal{A}=\{0,1,\cdots,2m-1\}$, we define $\varphi\in C(X)$ as follows: for any $x\in\mathcal{A}^{\mathbb{N}}$, if $x_0=2k$ for some $0\leq k\leq m-1$, then $\varphi(x)=1$; otherwise, $\varphi(x)=-1$.  We consider the set $Y$ consisted of all the elements $y\in\mathcal{A}^{\mathbb{N}}$ satisfying the following: for any $n\geq 0$, if $\vartheta_n\geq 0$, then $y_n=2k$ for some $0\leq k\leq m-1$; otherwise, $y_n=2k+1$ for some $0\leq k\leq m-1$. It can be checked that for any $y\in Y$ and $N\geq 1$, we always have $$\frac{1}{N}\left\vert\sum_{n=0}^{N-1}\vartheta_n\varphi(\sigma^ny)\right\vert=\frac{1}{N}\sum_{n=0}^{N-1}|\vartheta_n|.$$ Hence, $Y\subset \mathcal{N}_\vartheta(\sigma,\mathcal{A}^{\mathbb{N}})$. We denote $\mathcal{D}=\{0,1,\cdots,m-1\}$ and define a factor $\pi:(\mathcal{A}^{\mathbb{N}},\sigma)\to(\mathcal{D}^{\mathbb{N}},\sigma)$ by $\pi(x)_i=k$ if $x_i=2k\text{ or }2k+1$ for some $0\leq k\leq m-1$. Then we have $\pi(Y)=\mathcal{D}^{\mathbb{N}}$. By Lemma \ref{Lemma 2.3}, we have $$h_{\mathrm{top}}(\sigma,Y)\geq h_{\mathrm{top}}(\sigma,\pi(Y))= h_{\mathrm{top}}(\sigma,\mathcal{D}^{\mathbb{N}})=\log m.$$ As a result, $h_{\mathrm{top}}(\sigma,\mathcal{N}_\vartheta(\sigma,\mathcal{A}^{\mathbb{N}}))\geq\log m.$
\end{proof}
\subsection{Separated sets with recurrence} For dynamical system with the shadowing property, we show how to choose a large $(f;n,\varepsilon)$-separated set such that for any element $x$ in it, we have $d(x,f^nx)$ is small enough. 
\begin{mainlemma}\label{Lemma B}
	Suppose that $(X,f)$ is a dynamical system satisfying shadowing property and $\mu\in\mathcal{M}_f^e(X)$ with $h_\mu(f)>0$. Then for any $\eta<h_{\mu}(f)$, there exists $\varepsilon_*>0$ such that for any $0<\varepsilon\leq\varepsilon_*$ and any $0<\delta_*<1$, there exist $K_*,N_*\in\mathbb{N}$ with $N_*>K_*$ and $E\subset S_\mu$ with $\mu(E)>\delta_*$, such that for any $n\geq N_*$ and any $x\in E$, there exist $n-K_*\leq n_*\leq n$ and a $(f;n_*,\varepsilon_*)$-separated set $\Gamma_{n_*}$ satisfying the following
	\begin{enumerate}[(1)]
		\item $\Gamma_{n_*}\subset B(x,\varepsilon)\cap f^{-n_*}(B(x,\varepsilon))$;
		\item $\frac{\log|\Gamma_{n_*}|}{n_*}>\eta$.
	\end{enumerate}
\end{mainlemma}
\begin{proof}
	We choose $\tau>0$ such that $\eta+\tau<h_\mu(f)$. By Lemma \ref{Lemma 2.4}, there exist $\varepsilon_*>0$ and $N_1\in\mathbb{N}^+$ such that for any $n\geq N_1$,   there exists a $(f;n,  3\varepsilon_*)$-separated set $\Gamma'_n\subset  S_\mu$ with
	$$|\Gamma'_n|\geq e^{n(\eta+\tau)}.  $$ By the shadowing property, for any $0<\varepsilon\leq\varepsilon_*$, there exists $0<\delta<\varepsilon$ such that every $\delta$-pseudo orbit of $f$ can be $\frac{1}{2}\varepsilon$-traced by some point in $X$. Since $S_\mu$ is compact, there exists $m$ points $\{x_i\}_{i=1}^m$ of $S_\mu$ such that $$S_\mu\subset\bigcup_{i=1}^mB(x_i,\frac{\delta}{2}).$$ For $0<\delta_*<1$, since $(S_\mu,f)$ is topologically transitive, by Lemma \ref{Lemma 2.1}, there exist $E\subset S_\mu$ with $\mu(E)>\delta_*$ and $K\in\mathbb{N}$ such that for any $z\in E$ and any $1\leq i\leq m$, there exist $t_{0i},t_{i0}\in\mathbb{N}$ with $$0\leq t_{0i},t_{i0}\leq K$$ such that $$B(z,\frac{\delta}{2})\cap f^{-t_{0i}}(B(x_i,\frac{\delta}{2}))\neq\emptyset\text{ and }B(x_i,\frac{\delta}{2})\cap f^{-t_{i0}}(B(z,\frac{\delta}{2}))\neq\emptyset.$$ Denote $$K_*=2K\text{ and }N_*=\max\{N_1+K_*+1, \lfloor\frac{\eta K_*+2\log m}{\tau}\rfloor+K_*+1\}>K_*.$$ Then for any $n\geq N_*$, denote $\tilde{n}:=n-K_*$ we have $\tilde{n}\tau>\eta K_*+2\log m$, which means that $$\frac{\tilde{n}(\eta+\tau)-2\log m}{\tilde{n}+K_*}>\eta.$$ Since $\tilde{n}\geq N_1$,  there exists a $(f;\tilde{n},3\varepsilon_*)$-separated set $\Gamma'_{\tilde{n}}\subset  S_\mu$ with
	$$|\Gamma'_{\tilde{n}}|\geq e^{\tilde{n}(\eta+\tau)}.  $$ Now, we fix $x\in E$, then for any $1\leq i\leq m$, there exist $y_{0i}$, $y_{i0}\in S_\mu$ and $l_{0i},l_{i0}\in \mathbb{N}$ with $$0\leq l_{0i},l_{i0}\leq K$$ such that $$y_{0i}\in B(x,\frac{\delta}{2})\cap f^{-l_{0i}}(B(x_i,\frac{\delta}{2}))\text{ and }y_{i0}\in B(x_i,\frac{\delta}{2})\cap f^{-l_{i0}}(B(x,\frac{\delta}{2})).$$ Let $\mathcal{L}_{0i}$ be the orbital segment $$\mathcal{L}_{0i}:=<y_{0i},fy_{0i},\cdots,f^{l_{0i}-1}y_{0i}>$$ and $\mathcal{L}_{i0}$ be the orbital segment $$\mathcal{L}_{i0}:=<y_{i0},fy_{i0},\cdots,f^{l_{i0}-1}y_{i0}>.$$  Then, by the pigeonhole principle, there exist $1\leq i_1,i_2\leq m$ and $\Gamma\subset\Gamma'_{\tilde{n}}$ such that $$\Gamma\subset B(x_{i_1},\frac{\delta}{2})\cap f^{-{\tilde{n}}}(B(x_{i_2},\frac{\delta}{2}))\text{ and }|\Gamma|\geq\frac{|\Gamma'_{\tilde{n}}|}{m^2}.$$ For any $y\in\Gamma$, we define the following $\delta$-pseudo-orbit:$$
	\mathcal{L}_{y}:=\widetilde{\mathcal{L}}_{y}\widetilde{\mathcal{L}}_{y}\cdots\widetilde{\mathcal{L}}_{y}\cdots,\text{ where }\widetilde{\mathcal{L}}_{y}:=\mathcal{L}_{0i_1}<y,fy,\cdots,f^{\tilde{n}-1}y>\mathcal{L}_{i_20}.$$Denote $$n_*:=\tilde{n}+l_{0i_1}+l_{i_20},$$  then the length of $\widetilde{\mathcal{L}}_{y}$ is $n_*$ and $$n-K_*=\tilde{n}\leq n_*\leq \tilde{n}+K_*=n.$$ By the shadowing property, for each $y\in\Gamma$, we can find a $\theta_y\in X$ such that $\theta_y~\frac{1}{2}\varepsilon$-shadows $\mathcal{L}_y$. Denote $$\Gamma_{n_*}:=\{\theta_y: y\in\Gamma\}.$$ Since, $\frac{1}{2}\delta+\frac{1}{2}\varepsilon<\varepsilon$, it is clear that $$\Gamma_{n_*}\subset B(x,\varepsilon)\cap f^{-n_*}(B(x,\varepsilon)).$$ 
	
	Now, we will show that $\Gamma_{n_*}$ is $(f;n_*,\varepsilon_*)$-separated. In fact, for any $\theta_{y_1}, \theta_{y_2}\in\Gamma_{n_*}$ with $y_1\neq y_2\in\Gamma$,  we have that $d_{\tilde{n}}(f;y_1,y_2)>3\varepsilon_*$. This implies that $d_{n_*}(f;\theta_{y_1}, \theta_{y_2})>3\varepsilon_*-\varepsilon>\varepsilon_*$, which means that $\theta_{y_1}$ and $\theta_{y_2}$ are $(f;n_*,\varepsilon_*)$-separated. As a result, $\Gamma_{n_*}$ is $(f;n_*,\varepsilon_*)$-separated.
	
	Finally, we estimate $\frac{\log|\Gamma_{n_*}|}{n_*}$. We have that  $$\frac{\log|\Gamma_{n_*}|}{n_*}\geq\frac{\log|\Gamma|}{n_*}\geq\frac{\log|\Gamma'_{\tilde{n}}|-2\log m}{n_*}\geq\frac{\tilde{n}(\eta+\tau)-2\log m}{\tilde{n}+K_*}>\eta.$$
\end{proof}

\subsection{Points with large return times}
Given $x\in X$, $n\in\mathbb{N}^+$, $\varepsilon>0$ and a mistake function $g$, denote  $$R_n(f;x,\varepsilon):=\inf\{k\geq1:f^kx\in B_n(f;x,\varepsilon)\}$$ and $$R_n(f;g(n,\varepsilon);x,\varepsilon):=\inf\{k\geq1:f^kx\in B_n(f;g(n,\varepsilon);x,\varepsilon)\}.$$ It is clear that we always have $$R_n(f;x,\varepsilon)\geq R_n(f;g(n,\varepsilon);x,\varepsilon)\text{ for any mistake function }g(n,\varepsilon).$$ Given a nonempty subset $A\subset X$, denote $\operatorname{diam}(A):=\sup_{x,y\in A}d(x,y)$.
\begin{mainlemma}\label{Lemma C-New}
	Suppose that $(X,f)$ is a dynamical system and $g$ is a mistake function. Then for any $\mu\in\mathcal{M}_f^e(X)$ with $h_\mu(f)>0$ and any $0<\delta_\star<1$, there exists $\varepsilon_\star>0$ such that for any $0<\tilde{\varepsilon}\leq\varepsilon_\star$, any $A\subset X$ with $\mu(A)>\delta_\star$, any $L\in\mathbb{N}^+$, any $N\in\mathbb{N}^+$ and any $\varepsilon>0$, there exist $l\geq N$, $F\subset A$ with $\mu(F)>0$ and $\operatorname{diam}(F)<\varepsilon$, such that
	$$
		\mu(\{x\in F:f^{il}x\in F\text{ and }R_{l}(f;g(l,\tilde{\varepsilon});f^{il}x,\tilde{\varepsilon})>6Ll\text{ for any }0\leq i\leq 6L\})>0.$$ 
\end{mainlemma}
Directly from Lemma \ref{Lemma C-New}, we have
\begin{lemma}\label{Lemma C}
	Suppose that $(X,f)$ is a dynamical system. Then for any $\mu\in\mathcal{M}_f^e(X)$ with $h_\mu(f)>0$ and any $0<\delta_\star<1$, there exists $\varepsilon_\star>0$ such that for any $A\subset X$ with $\mu(A)>\delta_\star$, any $L\in\mathbb{N}^+$, any $N\in\mathbb{N}^+$ and any $\varepsilon>0$, there exist $l\geq N$, $F\subset A$ with $\mu(F)>0$ and $\operatorname{diam}(F)<\varepsilon$, such that $$\mu(\{x\in F:f^{il}x\in F\text{ and }R_{l}(f;f^{il}x,\varepsilon_\star)>6Ll\text{ for any }0\leq i\leq 6L\})>0.$$
\end{lemma}
Before we give the proof of Lemma \ref{Lemma C-New}, we introduce some useful results. The first result obtained by Rousseau et al. \cite{Rousseau-Varandas-Zhao-2012} shows the relationship between the metric entropy of an ergodic measure and the exponential growth rate of $R_n(f;g(n,\varepsilon);x,\varepsilon)$. 
\begin{lemma}\cite[Theorem A]{Rousseau-Varandas-Zhao-2012}\label{Lemma 0.1}
	Suppose that $(X,f)$ is a dynamical system and $g$ is a mistake function. Then for any $\mu\in\mathcal{M}_f^e(X)$, we have that the limit
	$$
	\underline{h}(f, x):=\lim _{\varepsilon \to 0} \liminf _{n \to \infty} \frac{1}{n} \log R_n(f;g(n,\varepsilon);x,\varepsilon)
	$$	
	exists for $\mu$-almost every $x$ and coincides with the metric entropy $h_\mu(f)$.
\end{lemma}
The second result is a well-known version of multiple recurrence theorem due to Furstenberg, for example, see \cite[Theorem 7.15]{Furstenberg-1981}.
\begin{lemma}[Furstenberg]\label{Lemma 0.2}
	Suppose that $(X,f)$ is a dynamical system. Then for any $\mu\in\mathcal{M}_f(X)$, any $A\subset X$ with $\mu(A)>0$ and any $k\in\mathbb{N}$, there exists $n\geq 1$ such that $$\mu(A\cap f^{-n}(A)\cap f^{-2n}(A)\cap\cdots\cap f^{-kn}(A))>0.$$
\end{lemma}

Now, we give the proof of Lemma \ref{Lemma C-New}.
\begin{proof}[Proof of Lemma \ref{Lemma C-New}]
	Since $\mu\in\mathcal{M}_f^e(X)$ and $g$ is a mistake function, by Lemma \ref{Lemma 0.1}, we have that $$\underline{h}(f, x)=\lim _{\varepsilon \to 0} \liminf _{n \to \infty} \frac{1}{n} \log R_n(f;g(n,\varepsilon);x,\varepsilon)=h_\mu(f)\text{ for } \mu\text{-almost every } x.$$ For any $m\in\mathbb{N}^+$, denote $$\varphi_m(x):=\inf_{0<\varepsilon\leq\frac{1}{m}}\liminf _{n \to \infty} \frac{1}{n} \log R_n(f;g(n,\varepsilon);x,\varepsilon),$$ then $$\lim_{m\to\infty}\varphi_m(x)=h_\mu(f)\text{ for } \mu\text{-almost every } x.$$ Choose $0<\tau<\frac{1}{3}\delta_\star$, then by Egoroff's theorem, there exists $A_1\subset X$ with $\mu(A_1)>1-\tau$ such that the sequence $\{\varphi_m\}_{m\geq1}$ converge uniformly to $h_\mu(f)$ on $A_1$. Choose $0<\delta<\frac{1}{2}h_\mu(f)$, then there exists $N_1\geq N$, such that for any $m\geq N_1$ and any $x\in A_1$, one has $|\varphi_m(x)-h_\mu(f)|<\delta$, hence, $$\varphi_m(x)=\inf_{0<\varepsilon\leq\frac{1}{m}}\lim_{n\to\infty}\inf_{q\geq n}\frac{1}{q}\log R_q(f;g(q,\varepsilon);x,\varepsilon)>h_\mu(f)-\delta.$$ Denote $\varepsilon_\star=\frac{1}{N_1}>0$, given $0<\tilde{\varepsilon}\leq\varepsilon_\star$, again, by Egoroff's theorem, there exists $A_2\subset A_1$ with $\mu(A_2)>1-2\tau$, such that the sequence $$\{\inf_{q\geq n}\frac{1}{q}\log R_q(f;g(q,\tilde{\varepsilon});x,\tilde{\varepsilon})\}_{n\geq1}$$ converge uniformly on $A_2$. Hence, there exists $N_2\geq N_1$, such that for any $x\in A_2$, one has $$\inf_{q\geq N_2}\frac{1}{q}\log R_q(f;g(q,\tilde{\varepsilon});x,\tilde{\varepsilon})>h_\mu(f)-2\delta>0.$$ Hence, there exists $N_3\geq N_2$, such that for any $x\in A_2$, we have $$\inf_{r\geq N_3}R_{r}(f;g(r,\tilde{\varepsilon});x,\tilde{\varepsilon})>6Lr.$$ Given $A\subset X$ with $\mu(A)>\delta_\star$, denote $A_3=A\cap A_2$, then $$\mu(A_3)=\mu(A)+\mu(A_2)-\mu(A\cup A_2)>\mu(A)-2\tau>0.$$
	By the regularity of $\mu$, there exists a compact subset $A_4\subset A_3$ such that $$\mu(A_4)>\mu(A_3)-\tau>\mu(A)-3\tau>0.$$
	Since $A_4$ is compact, for any $\varepsilon>0$, there exist $m$ points $\{x_i\}_{i=1}^m$ of $A_4$ such that $$A_4\subset\bigcup_{i=1}^mB(x_i,\frac{\varepsilon}{2}).$$ Hence, there exists $1\leq j\leq m$ with $\mu(A_4\cap B(x_j,\frac{\varepsilon}{2}))>0$. Denote $$F:=A_4\cap B(x_j,\frac{\varepsilon}{2}),$$ then $$F\subset A_4\subset A\cap A_2,~\mu(F)>0\text{ and }\operatorname{diam}(F)<\varepsilon.$$
	By using Lemma \ref{Lemma 0.2} for $F$ and $k=6LN_3$, there exists $n\geq 1$ such that $$\mu(F\cap f^{-n}(F)\cap f^{-2n}(F)\cap\cdots\cap f^{-6LN_3n}(F))>0.$$ Denote $l=N_3n$, then $$\mu(F\cap f^{-l}(F)\cap f^{-2l}(F)\cap\cdots\cap f^{-6Ll}(F))>0.$$
	As a result, $$\mu(\{x\in F:f^{il}x\in F\text{ and }R_{l}(f;g(l,\tilde{\varepsilon});f^{il}x,\tilde{\varepsilon})>6Ll\text{ for any }0\leq i\leq 6L\})>0.$$
	
\end{proof}

\section{Proof of Theorem \ref{Theorem B}}\label{Section 5}
In this section, we prove Theorem \ref{Theorem B}. First of all, we prove the following theorem.
\begin{theorem}\label{Theorem 4.1}
	Suppose that $(X,f)$ is a dynamical system and $L\in\mathbb{N}^+$, if there exists a subsystem $(Y,f^L)$ of $(X,f^L)$ such that
	\begin{enumerate}[(1)]
		\item $(Y,f^L)$ satisfies the shadowing property with $h_{\mathrm{top}}(f^L,Y)>0$;
		\item $X=\bigcup_{i=0}^{L-1}f^i(Y)$.
	\end{enumerate}
	Then $h_{\mathrm{top}}(f,X)=\frac{1}{L}h_{\mathrm{top}}(f^L,Y)>0$ and for any $0<\eta<h_{\mathrm{top}}(f,X)$ and any $N\in\mathbb{N}^+$, there exist $m,k\in\mathbb{N}^+$ with $$k\in\{Ln:n\in\mathbb{N}^+\}, k\geq N \text{ and } \frac{1}{k}\log m>\eta,$$ an alphabet $\mathcal{A}$ with $|\mathcal{A}|=m$ and a subsystem $(\Lambda,f^k)$ of $(Y,f^k)$ such that 
	\begin{enumerate}[(1)]
		\item $f^i(\Lambda)\cap f^j(\Lambda)=\emptyset$ for any $0\leq i<j\leq k-1$;
		\item there is a factor $\pi: (\Lambda,   f^k)\to (\mathcal{A}^{\mathbb{N}},  \sigma)$.
	\end{enumerate}
\end{theorem}
\begin{proof}
	By Lemma \ref{Lemma 2.2}, we have that $h_{\mathrm{top}}(f,X)=\frac{1}{L}h_{\mathrm{top}}(f^L,Y)>0$. Given $0<\eta<h_{\mathrm{top}}(f,X)$, we choose $\tau>0$ such that $$\eta L+4\tau<h_{\mathrm{top}}(f^L,Y).$$ By the variational principle, there exists $\mu\in\mathcal{M}_{f^L}^e(Y)$ such that $$h_\mu(f^L)>\eta L+3\tau.$$ 
	
	Since $(Y,f^L)$ satisfying the shadowing property, by Lemma \ref{Lemma B}, there exists $\varepsilon_*>0$ such that for any $0<\varepsilon_1\leq\varepsilon_*$ and any $0<\delta_*<1$, there exist $K_*,N_*\in\mathbb{N}$ with $N_*>K_*$ and $E\subset S_\mu$ with $\mu(E)>\delta_*$, such that for any $n\geq N_*$ and any $x\in E$, there exist $n-K_*\leq n_*\leq n$ and a $(f^L;n_*,\varepsilon_*)$-separated set $\Gamma_{n_*}(f^L)$ satisfying 
	\begin{equation*}
		\Gamma_{n_*}(f^L)\subset B(x,\varepsilon_1)\cap f^{-Ln_*}(B(x,\varepsilon_1))\cap Y
	\end{equation*}
	and
	\begin{equation*}
		\frac{\log|\Gamma_{n_*}(f^L)|}{n_*}>\eta L+3\tau.
	\end{equation*}

	We define $$\nu:=\frac{1}{L}\sum_{i=0}^{L-1}\mu\circ f^{-i},$$ then for any measurable set $A\subset X$ with $f^{-1}A=A$, we have $$f^{-L}(A)=A\text{ and }\nu(A)=\frac{1}{L}\sum_{i=0}^{L-1}\mu\circ f^{-i}(A)=\mu(A).$$ Hence, $\mu(A)=0$ or $1$ implies $\nu(A)=0$ or $1$. Since $\mu\in\mathcal{M}_{f^L}^e(X)$, we have that $\nu\in\mathcal{M}_f^e(X)$. By \cite[Lemma 18.5]{DGS}, we have $h_\nu(f^L)=h_\mu(f^L)>0$, hence, $$h_\nu(f)=\frac{1}{L}h_\nu(f^L)>0.$$
	
	We set $\delta_\star=\frac{1}{2L}$. By Lemma \ref{Lemma C}, there exists $\varepsilon_\star>0$ such that for any $A\subset X$ with $\nu(A)>\delta_\star$, any $N_1\in\mathbb{N}^+$ and any $\varepsilon_2>0$, there exist $l\geq N_1$, $F\subset A$ with $\nu(F)>0$ and $\operatorname{diam}(F)<\varepsilon_2$, such that
	\begin{equation}\label{equ (4.1)}
	\nu(\{x\in F:f^{il}x\in F\text{ and }R_{l}(f;f^{il}x,\varepsilon_\star)>6Ll\text{ for any }0\leq i\leq 6L\})>0.
	\end{equation}
	
	We fix $0<\varepsilon\leq\min\{\frac{1}{6}\varepsilon_\star,\frac{1}{6}\varepsilon_*\}$ such that for any $x,y\in X$ with $d(x,y)<\varepsilon$, we always have 
	$$d_L(f;x,y)<\min\{\frac{1}{6}\varepsilon_\star,\frac{1}{6}\varepsilon_*\}.$$ By the shadowing property for $(Y,f^L)$, there exists $0<\delta<\varepsilon$ such that every $\delta$-pseudo orbit of $f^L$ in $Y$ is $\varepsilon$-traced by some point in $Y$ for $f^L$. 
	
	We set $\varepsilon_1=\frac{1}{2}\delta$ and $\delta_*=\frac{1}{2}$, get $K_*,N_*\in\mathbb{N}^+$ with $N_*>K_*$ and $E\subset S_\mu$ with $\mu(E)>\delta_*$. 
	Now, we choose $N_1\in\mathbb{N}^+$ large enough such that for any $n\geq N_1-K_*$, we have   $$n\geq\max\{N,N_*\}$$ and 
	\begin{equation}\label{equ (4.2)}
		e^{n(\eta L+3\tau)}>e^{n(\eta L+2\tau)}+6L(n+K_*).
	\end{equation}
	From the definition of $\nu$, we have $$\nu(E)\geq\frac{1}{L}\mu(E)>\frac{1}{L}\delta_*=\frac{1}{2L}=\delta_\star.$$ 
	We set $\varepsilon_2=\frac{1}{2}\delta$, get $l\geq N_1$, $F\subset E$ with $\nu(F)>0$ and $\operatorname{diam}(F)<\frac{1}{2}\delta$ satisfying equation (\ref{equ (4.1)}). In particular, we can choose $x_0\in F$ such that 
	\begin{equation}\label{equ (4.3)}
		d(x_0,f^{il}x_0)<\frac{1}{2}\delta\text{ and }R_{l}(f;f^{il}x_0,\varepsilon_\star)>6Ll\text{ for any }0\leq i\leq 6L.
	\end{equation}
	
	\textbf{Claim 1.} For any $1\leq j\leq (6L-1)l$, we always have $$d(f^ix_0,f^{j+i}x_0)\geq\varepsilon_\star\text{ for some }0\leq i\leq l-1.$$
	\begin{proof}[Proof of Claim 1]
		On the contrary, we assume that $d(f^ix_0,f^{j+i}x_0)<\varepsilon_\star$ for any $0\leq i\leq l-1$. As a result, we have that $f^jx_0\in B_{l}(f;x_0,\varepsilon_\star)$. Hence, $R_l(f;x_0,\varepsilon_\star)\leq (6L-1)l$. However, by equation (\ref{equ (4.3)}), we have $R_l(f;x_0,\varepsilon_\star)> 6Ll$, it is a contradiction. As a result, we have $$d(f^ix_0,f^{j+i}x_0)\geq\varepsilon_\star\text{ for some }0\leq i\leq l-1.$$
	\end{proof}
	
	Now, for $n=l$, $x=x_0\in F\subset E$, we get $$l-K_*\leq n_*\leq l$$ and a $(f^L;n_*,\varepsilon_*)$-separated set $\Gamma_{n_*}(f^L)$ satisfying the following:
	\begin{equation}\label{equ (4.4)}
		\Gamma_{n_*}(f^L)\subset B(x_0,\frac{1}{2}\delta)\cap f^{-Ln_*}(B(x_0,\frac{1}{2}\delta))\cap Y;
	\end{equation}
	
	\begin{equation}\label{equ (4.5)}
		\frac{\log|\Gamma_{n_*}(f^L)|}{n_*}>\eta L+3\tau.
	\end{equation}
	Since $n_*\geq l-K_*\geq N_1-K_*$ and $l\leq n_*+K_*$, by equation (\ref{equ (4.2)}), we have 
	\begin{equation*}
		e^{n_*(\eta L+3\tau)}>e^{n_*(\eta L+2\tau)}+6L(n_*+K_*)\geq e^{n_*(\eta L+2\tau)}+6Ll.
	\end{equation*}
	Combining with equation (\ref{equ (4.5)}), we have that
	\begin{equation}\label{equ (4.6)}
		\frac{1}{n_*}\log(|\Gamma_{n_*}(f^L)|-6Ll)>\eta L+2\tau.
	\end{equation}
	
	Denote
	\begin{multline*}
		G:=\{x\in\Gamma_{n_*}(f^L): d(f^ix,f^{j+i}x_0)\leq\frac{1}{2}\varepsilon_*\text{ for some }\\ 0\leq j\leq 6Ll-Ln_*\text{ and any }0\leq i\leq Ln_*-1\}.
	\end{multline*} 
	Since $\Gamma_{n_*}(f^L)$ is $(f^L;n_*,\varepsilon_*)$-separated, we have $$|G|\leq 6Ll.$$
	Denote $$\Gamma'_{n_*}(f^L):=\Gamma_{n_*}(f^L)\setminus G,$$ then combining with equation (\ref{equ (4.6)}), we have 
	\begin{equation}\label{equ (4.7)}
		\frac{1}{n_*}\log(|\Gamma'_{n_*}(f^L)|)>\eta L+2\tau.
	\end{equation}
	And for any $x\in \Gamma'_{n_*}(f^L)$ and any $0\leq j\leq 6Ll-Ln_*$, we always have that
	\begin{equation}\label{equ (4.8)}
		d(f^ix,f^{j+i}x_0)>\frac{1}{2}\varepsilon_*\text{ for some }0\leq i\leq Ln_*-1.
	\end{equation}
	
	We choose $p\in\mathbb{N}^+$ large enough such that 
	\begin{equation}\label{equ (4.9)}
		\frac{pn_*(\eta L+2\tau)}{6l+pn_*}>\eta L+\tau
	\end{equation} and denote $$k=6Ll+pLn_*=L(6l+pn_*)\geq N_1\geq N\text{ and }m=|\Gamma'_{n_*}(f^L)|^p.$$ Then combining equation (\ref{equ (4.7)}) and equation (\ref{equ (4.9)}), we have that $$\frac{1}{k}\log m>\eta+\frac{1}{L}\tau>\eta.$$
	
	Now we define an alphabet $$\mathcal{A}:=\Gamma'_{n_*}(f^L)^p=\Gamma'_{n_*}(f^L)\times\Gamma'_{n_*}(f^L)\times\cdots\times\Gamma'_{n_*}(f^L)\text{ with }|\mathcal{A}|=m.$$ For a element $\theta=(\theta^1,\theta^2,\cdots,\theta^p)\in\mathcal{A}$, we define $$\widetilde{\mathcal{L}}_{\theta}(f^L):=\mathcal{L}_{x_0}(f^L)\mathcal{L}'_{\theta^1}(f^L)\mathcal{L}'_{\theta^2}(f^L)\cdots\mathcal{L}'_{\theta^p}(f^L),$$where $$\mathcal{L}_{x_0}(f^L):=<x_0,f^Lx_0,\cdots,f^{L(6l-1)}x_0>$$ and $$\mathcal{L}'_{\theta^i}(f^L):=<\theta^i,f^L\theta^i,\cdots,f^{L(n_*-1)}\theta^i>\text{ for any }1\leq i\leq p.$$
	
	$$\widetilde{\mathcal{L}}_{\theta}(f):=\mathcal{L}_{x_0}(f)\mathcal{L}'_{\theta^1}(f)\mathcal{L}'_{\theta^2}(f)\cdots\mathcal{L}'_{\theta^p}(f),$$where $$\mathcal{L}_{x_0}(f):=<x_0,fx_0,\cdots,f^{6Ll-1}x_0>$$ and $$\mathcal{L}'_{\theta^i}(f):=<\theta^i,f\theta^i,\cdots,f^{Ln_*-1}\theta^i>\text{ for any }1\leq i\leq p.$$
	Then the length of $\widetilde{\mathcal{L}}_{\theta}(f)$ is $k$.
	
	Given $\underline{\theta}=\theta_0\theta_1\cdots\in\mathcal{A}^{\mathbb{N}}$, we define $$\mathcal{L}_{\underline{\theta}}(f^L):=\widetilde{\mathcal{L}}_{\theta_0}(f^L)\widetilde{\mathcal{L}}_{\theta_1}(f^L)\cdots\text{ and }\mathcal{L}_{\underline{\theta}}(f):=\widetilde{\mathcal{L}}_{\theta_0}(f)\widetilde{\mathcal{L}}_{\theta_1}(f)\cdots.$$
	By equation (\ref{equ (4.3)}) and equation (\ref{equ (4.4)}), it can be checked that $\mathcal{L}_{\underline{\theta}}(f^L)$ is a $\delta$-pseudo-orbit of $f^L$ in $Y$. Hence, it can be $\varepsilon$-traced by some point in $Y$ for $f^L$. Recall that for any $x,y\in X$ with $d(x,y)<\varepsilon$, we always have 
	$$d_L(f;x,y)<\hat{\varepsilon}:=\min\{\frac{1}{6}\varepsilon_\star,\frac{1}{6}\varepsilon_*\}.$$ Hence, if we rewrite $\mathcal{L}_{\underline{\theta}}(f)$ as $\mathcal{L}_{\underline{\theta}}(f)=\omega_0\omega_1\cdots$, then the set $$Z_{\underline{\theta}}:=\{z\in Y:d(f^iz,\omega_i)\leq\hat{\varepsilon}\text{ for any }i\in\mathbb{N}\}$$ is nonempty and closed. We define $$\Lambda:=\bigcup_{\underline{\theta}\in\mathcal{A}^\mathbb{N}}Z_{\underline{\theta}},$$ then we have $f^k(\Lambda)\subset \Lambda\subset Y$.
	
	\textbf{Claim 2.} We claim that $Z_{\underline{\theta}}\cap Z_{\underline{\hat{\theta}}}=\emptyset$ for any $\underline{\theta}\neq\underline{\hat{\theta}}\in\mathcal{A}^\mathbb{N}$.
	\begin{proof}[Proof of Claim 2]
		Since $\underline{\theta}\neq\underline{\hat{\theta}}$, there exist $i\in\mathbb{N}$ and $1\leq j\leq p$ such that $\theta_i^{j}\neq\hat{\theta}_i^j\in\Gamma'_{n_*}(f^L).$ Suppose that $\omega_r=\theta_i^{j}$ and $\hat{\omega}_r=\hat{\theta}_i^j$ for some $0\leq r\leq (i+1)k-Ln_*$. Since $\omega_r$ and $\hat{\omega}_r$ are $(f^L;n_*,\varepsilon_*)$-separated, for any $z\in Z_{\underline{\theta}}$ and $\hat{z}\in Z_{\underline{\hat{\theta}}}$, we have $$d(f^{r+t}z,f^{r+t}\hat{z})>\varepsilon_*-2\hat{\varepsilon}\geq\frac{2}{3}\varepsilon_*>0\text{ for some }0\leq t\leq Ln_*-1.$$ Hence, $z\neq\hat{z}$, as a result, $Z_{\underline{\theta}}\cap Z_{\underline{\hat{\theta}}}=\emptyset$.
	\end{proof}

	\textbf{Claim 3.} We claim that $\Lambda$ is closed and hence $(\Lambda,f^k)$ is a subsystem of $(X,f^k)$.
	\begin{proof}[Proof of Claim 3]
		Since $f^k(\Lambda)\subset \Lambda$, it is enough to show that $\Lambda$ is closed. 	Given $\underline{\theta}=\theta_0\theta_1\cdots\in\mathcal{A}^{\mathbb{N}}$ and $r\in\mathbb{N}$, we define $$\mathcal{L}_{\underline{\theta}}(f,r):=\widetilde{\mathcal{L}}_{\theta_0}(f)\widetilde{\mathcal{L}}_{\theta_1}(f)\cdots\widetilde{\mathcal{L}}_{\theta_r}(f)=\omega_0\omega_1\cdots\omega_{(r+1)k-1}$$ and $$Z_{\underline{\theta},r}:=\{z\in Y:d(f^iz,\omega_i)\leq\hat{\varepsilon}\text{ for any }0\leq i\leq (r+1)k-1\}.$$ Then $Z_{\underline{\theta},r}$ is nonempty, closed and for any $\underline{\theta},\underline{\hat{\theta}}\in\mathcal{A}^\mathbb{N}$ with $\theta_0\theta_1\cdots\theta_r=\hat{\theta}_0\hat{\theta}_1\cdots\hat{\theta}_r$, we always have that $Z_{\underline{\theta},r}=Z_{\underline{\hat{\theta}},r}$. As a result, $$\Lambda=\bigcup_{\underline{\theta}\in\mathcal{A}^\mathbb{N}}Z_{\underline{\theta}}=\bigcap_{r\geq 0}\bigcup_{\underline{\theta}\in\mathcal{A}^\mathbb{N}}Z_{\underline{\theta},r}$$ is closed.
	\end{proof}
	
	Now, we define a map $\pi: (\Lambda,   f^k)\to (\mathcal{A}^{\mathbb{N}},  \sigma)$ as $\pi(z):=\underline{\theta}$ for any $z\in Z_{\underline{\theta}}$. Then $\pi\circ f^k=\sigma\circ\pi$. By Claim 2, $\pi$ is well-defined. Since $Z_{\underline{\theta}}\neq\emptyset$ for any $\underline{\theta}\in\mathcal{A}^\mathbb{N}$, we conclude that $\pi$ is surjective. 
	
	\textbf{Claim 4.} We claim that $\pi$ is continuous.
	\begin{proof}[Proof of Claim 4]
		It is enough to show that for any $r\in\mathbb{N}$, any $z\in Z_{\underline{\theta}}$ and $\hat{z}\in Z_{\underline{\hat{\theta}}}$ with $d(f^iz,f^i\hat{z})<\frac{2}{3}\varepsilon_*$ for any $0\leq i\leq (r+1)k-1$, we always have that $\theta_0\theta_1\cdots\theta_r=\hat{\theta}_0\hat{\theta}_1\cdots\hat{\theta}_r$. We assume that there exist $0\leq s\leq r$ and $1\leq j\leq p$ such that $\theta_s^j\neq \hat{\theta}_s^j\in\Gamma'_{n_*}(f^L)$, then follow the same line of the proof of Claim 2, we conclude that $$d(f^tz,f^t\hat{z})>\frac{2}{3}\varepsilon_*\text{ for some }0\leq t\leq (r+1)k-1.$$ It is a contradiction to $d(f^iz,f^i\hat{z})<\frac{2}{3}\varepsilon_*$ for any $0\leq i\leq (r+1)k-1$.
	\end{proof}
	As a result, $\pi$ is factor. Finally, we show that $f^i(\Lambda)\cap f^j(\Lambda)=\emptyset$ for any $0\leq i<j\leq k-1$. When $k=1$, it is clear, hence, in the following, we assume that $k\geq2$. 
	
	\textbf{Claim 5.} We claim that it is enough to show that $\Lambda\cap f^r(\Lambda)=\emptyset$ for any $1\leq r\leq k-1$.
	\begin{proof}[Proof of Claim 5.]
		If $f^i(\Lambda)\cap f^j(\Lambda)\neq\emptyset$ for some $0\leq i<j\leq k-1$, denote $r=j-i$, then $$\emptyset\neq f^{k-i}(f^{i}(\Lambda)\cap f^{j}(\Lambda))\subset f^k(\Lambda)\cap f^{j+k-i}(\Lambda)\subset\Lambda\cap f^{j-i}(\Lambda)=\Lambda\cap f^r(\Lambda).$$ 
	\end{proof}
	Suppose that $\Lambda\cap f^{r}(\Lambda)\neq\emptyset$ for some $1\leq r\leq k-1$. Then for any $z\in\Lambda\cap f^{r}(\Lambda)$, there exist $\underline{\theta},\underline{\hat{\theta}}\in\mathcal{A}^\mathbb{N}$ such that $z\in Z_{\underline{\theta}}\cap f^{r}(Z_{\underline{\hat{\theta}}})$. Hence, we have that $$d(f^iz,\omega_i)\leq\hat{\varepsilon}\text{ and }d(f^iz,\hat{\omega}_{r+i})\leq\hat{\varepsilon},\text{ for any }i\in\mathbb{N}.$$ 
	As a result,
	\begin{equation}\label{equ (4.10)}
		d(\omega_i,\hat{\omega}_{r+i})\leq2\hat{\varepsilon}=\min\{\frac{1}{3}\varepsilon_\star,\frac{1}{3}\varepsilon_*\},\text{ for any }i\in\mathbb{N}.
	\end{equation} 
	Recall that $n_*\leq l$,
	\begin{equation}\label{equ (4.11)}
		\omega_{kj+i}=\hat{\omega}_{kj+i}=f^ix_0\text{ for any }j\in\mathbb{N}\text{ and any }0\leq i\leq 6Ll-1,
	\end{equation}
	\begin{equation}\label{equ (4.12)}
		\omega_{6Ll+tLn_*},\hat{\omega}_{6Ll+tLn_*}\in\Gamma'_{n_*}(f^L)\text{ for any }0\leq t\leq p-1.
	\end{equation}
	
	If $1\leq r\leq (6L-1)l$, then $\hat{\omega}_{r+i}=f^{r+i}x_0$ for any $0\leq i\leq l-1$. By equation (\ref{equ (4.10)}), we have $$d(f^ix_0,f^{r+i}x_0)\leq\frac{1}{3}\varepsilon_\star\text{ for any }0\leq i\leq l-1.$$ It is a contradiction to Claim 1.
	
	If $(6L-1)l< r< k-(6L-1)l$, then by the fact that $Ln_*\leq Ll$, equation (\ref{equ (4.10)}) and equation (\ref{equ (4.12)}), there always exist $0\leq t\leq p-1$ such that $\hat{\omega}_{6Ll+tLn_*}\in\Gamma'_{n_*}(f^L)$ and $0\leq j\leq 4Ll$ such that $$d(\omega_{j+i},\hat{\omega}_{6Ll+tLn_*+i})\leq\frac{1}{3}\varepsilon_*\text{ for any }0\leq i\leq Ln_*-1,$$ which means that $$d(f^{j+i}x_0,f^{i}\hat{\omega}_{6Ll+tLn_*})\leq\frac{1}{3}\varepsilon_*\text{ for any }0\leq i\leq Ln_*-1.$$ It is a contradiction to equation (\ref{equ (4.8)}).
	
	If $k-(6L-1)l\leq r\leq k-1$, then $1\leq k-r\leq(6L-1)l$. By equation (\ref{equ (4.10)}), we have $$d(\omega_{k-r+i},\hat{\omega}_{k+i})\leq\frac{1}{3}\varepsilon_\star\text{ for any }0\leq i\leq l-1.$$ By equation (\ref{equ (4.11)}), we have that $$d(f^{k-r+i}x_0,f^ix_0)\leq\frac{1}{3}\varepsilon_\star\text{ for any }0\leq i\leq l-1.$$ It is a contradiction to Claim 1.
	
	In conclusion, we have that $\Lambda\cap f^r(\Lambda)=\emptyset$ for any $1\leq r\leq k-1$.
	
	Now, we have finished the proof of Theorem \ref{Theorem 4.1}. 
\end{proof}
    Let us recall a result obtained by Fan et al. in the proof of \cite[Theorem 3.3]{Fan-Fan-Ryzhikov-Shen2022}. 
    \begin{lemma}\cite{Fan-Fan-Ryzhikov-Shen2022}\label{Lemma 4.2}
    	Suppose that $(X,f)$ is a dynamical system, $k\in\mathbb{N}^+$ and $(\Lambda,f^k)$ is a subsystem of $(X,f^k)$ such that $f^i(\Lambda)\cap f^j(\Lambda)=\emptyset$ for any $0\leq i<j\leq k-1$. Then for any non-trivial weight $\vartheta\in \ell^\infty(\mathbb{N})$, there exists $0\leq\tau\leq k-1$, such that $$\varpi:=(\varpi_n)_{n\geq 0},\text{ where }\varpi_0=0\text{ and }\varpi_{n}=\vartheta_{k(n-1)+\tau}\text{ for any }n\geq 1,$$ is a non-trivial weight in $\ell^\infty(\mathbb{N})$ and  $$f^{k-\tau}(\mathcal{N}_\varpi(f^k,\Lambda))\subset\mathcal{N}_\vartheta(f,X).$$
    \end{lemma}
\begin{proof}[Sketch of proof]
	For any non-trivial weight $\vartheta\in \ell^\infty(\mathbb{N})$, we can always find $0\leq \tau\leq k-1$ such that $$\limsup_{N\to\infty}\frac{1}{N}\sum_{n=0}^{N-1}|\vartheta_{kn+\tau}|>0.$$  For $x\in\mathcal{N}_\varpi(f^k,\Lambda)$, there exists $\varphi\in C(\Lambda)$ with $$\limsup_{N\to\infty}\frac{1}{N}\left\vert\sum_{n=1}^{N-1}\vartheta_{k(n-1)+\tau}\varphi(f^{k(n-1)+\tau}(f^{k-\tau}x))\right\vert>0.$$ Denote $Y:=\bigcup_{i=0}^{k-1}f^i(\Lambda)$, since $f^i(\Lambda)\cap f^j(\Lambda)=\emptyset$ for any $0\leq i<j\leq k-1$, we can define $\tilde{\varphi}\in C(Y)$ as $\tilde{\varphi}|_\Lambda=\varphi$ and $\tilde{\varphi}(\bigcup_{i=1}^{k-1}f^i(\Lambda))=\{0\}.$ By Tietze extension theorem, we can find $\varphi^*\in C(X)$ with $\varphi^*|_{Y}=\tilde{\varphi}$. For $y=f^{k-\tau}x$, it can be checked that $$\limsup_{N\to\infty}\frac{1}{kN}\left\vert\sum_{n=0}^{kN-1}\vartheta_n\varphi^*(f^ny)\right\vert>0.$$
\end{proof}
\begin{lemma}\label{Lemma 4.3}
		Suppose that $\pi:(X,f)\to(Y,g)$ is a factor between two dynamical systems. Then for any non-trivial weight $\vartheta\in \ell^\infty(\mathbb{N})$, we have $$\pi^{-1}(\mathcal{N}_\vartheta(g,Y))\subset\mathcal{N}_\vartheta(f,X).$$ As a result, $$h_{\mathrm{top}}(f,\mathcal{N}_\vartheta(f,X))\geq h_{\mathrm{top}}(g,\mathcal{N}_\vartheta(g,Y)).$$
\end{lemma}
\begin{proof}
	For any $\varphi\in C(Y)$, we have $\hat{\varphi}:=\varphi\circ\pi\in C(X)$ and $$\hat{\varphi}(f^nx)=\varphi(g^n(\pi x))\text{ for any }x\in X\text{ and }n\in\mathbb{N}.$$ Hence, for any $y\in \mathcal{N}_\vartheta(g,Y)$ and $\varphi\in C(Y)$ with  $$\limsup_{N\to\infty}\frac{1}{N}\left\vert\sum_{n=0}^{N-1}\vartheta_n\varphi(g^ny)\right\vert>0,$$ we always have $$\limsup_{N\to\infty}\frac{1}{N}\left\vert\sum_{n=0}^{N-1}\vartheta_n\hat{\varphi}(f^nx)\right\vert>0\text{ for any }x\in \pi^{-1}(\{y\}).$$ As a result, $$\pi^{-1}(\mathcal{N}_\vartheta(g,Y))\subset\mathcal{N}_\vartheta(f,X).$$ Finally, by Lemma \ref{Lemma 2.3}, we have 
	\begin{multline*}
		h_{\mathrm{top}}(f,\mathcal{N}_\vartheta(f,X))\geq h_{\mathrm{top}}(f,\pi^{-1}(\mathcal{N}_\vartheta(g,Y)))\\\geq h_{\mathrm{top}}(g,\pi(\pi^{-1}(\mathcal{N}_\vartheta(g,Y))))= h_{\mathrm{top}}(g,\mathcal{N}_\vartheta(g,Y)).
	\end{multline*}
\end{proof}

Based on Lemma \ref{Lemma A}, Theorem \ref{Theorem 4.1}, Lemma \ref{Lemma 4.2} and Lemma \ref{Lemma 4.3}, we give the proof of Theorem \ref{Theorem B}.
\begin{proof}[Proof of Theorem \ref{Theorem B}]
	By Theorem \ref{Theorem 4.1}, we have $h_{\mathrm{top}}(f,X)=\frac{1}{L}h_{\mathrm{top}}(f^L,Y)>0$. Given $0<\eta<h_{\mathrm{top}}(f,X)$, we choose $\tau>0$ with $\eta+\tau<h_{\mathrm{top}}(f,X)$ and choose an integer $N_1\geq4$ large enough such that for any $n\geq N_1$, we have $$\frac{\log\lfloor\frac{1}{2}n\rfloor}{\log n}>\frac{\eta}{\eta+\tau}.$$ Denote $N=\lfloor\frac{\log N_1}{\eta+\tau}\rfloor+1$. Then by Theorem \ref{Theorem 4.1}, there exist $m,k\in\mathbb{N}^+$ with 
	\begin{equation}\label{equ (4.13)}
		k\in\{Ln:n\in\mathbb{N}^+\}, k\geq N \text{ and } \frac{1}{k}\log m>\eta+\tau,
	\end{equation} an alphabet $\mathcal{A}$ with $|\mathcal{A}|=m$ and a subsystem $(\Lambda,f^k)$ of $(Y,f^k)$ such that 
	\begin{enumerate}[(1)]
		\item $f^i(\Lambda)\cap f^j(\Lambda)=\emptyset$ for any $0\leq i<j\leq k-1$;
		\item there is a factor $\pi: (\Lambda,   f^k)\to (\mathcal{A}^{\mathbb{N}},  \sigma)$.
	\end{enumerate}
	As a result, 
	\begin{equation}\label{equ (4.14)}
		m>e^{k(\eta+\tau)}\geq e^{\log N_1}=N_1\text{ and thus }\frac{\log\lfloor\frac{1}{2}m\rfloor}{\log m}>\frac{\eta}{\eta+\tau}.
	\end{equation}
	Now, given a non-trivial weight $\vartheta\in \ell^\infty(\mathbb{N})$, since $f^i(\Lambda)\cap f^j(\Lambda)=\emptyset$ for any $0\leq i<j\leq k-1$, by Lemma \ref{Lemma 4.2}, there exist $0\leq\tau\leq k-1$ and a non-trivial weight $\varpi\in \ell^\infty(\mathbb{N})$ such that $$f^{k-\tau}(\mathcal{N}_\varpi(f^k,\Lambda))\subset\mathcal{N}_\vartheta(f,X).$$ Hence, by Lemma \ref{Lemma 2.2}, we have
	\begin{multline}\label{equ (4.15)}
		h_{\mathrm{top}}(f,\mathcal{N}_\vartheta(f,X))\geq h_{\mathrm{top}}(f,f^{k-\tau}(\mathcal{N}_\varpi(f^k,\Lambda)))\\=h_{\mathrm{top}}(f,\mathcal{N}_\varpi(f^k,\Lambda))=\frac{1}{k}h_{\mathrm{top}}(f^k,\mathcal{N}_\varpi(f^k,\Lambda)).
	\end{multline}
	By Lemma \ref{Lemma 4.3}, we have
	\begin{equation}\label{equ (4.16)}
		h_{\mathrm{top}}(f^k,\mathcal{N}_{\varpi}(f^k,\Lambda))\geq h_{\mathrm{top}}(\sigma,\mathcal{N}_{\varpi}(\sigma,\mathcal{A}^{\mathbb{N}})).
	\end{equation}
	By Lemma \ref{Lemma A}, we have that
	\begin{equation}\label{equ (4.17)}
		h_{\mathrm{top}}(\sigma,\mathcal{N}_{\varpi}(\sigma,\mathcal{A}^{\mathbb{N}}))\geq\log \lfloor\frac{1}{2}m\rfloor.
	\end{equation}
	Combining equations (\ref{equ (4.13)})-(\ref{equ (4.17)}), we have that $$h_{\mathrm{top}}(f,\mathcal{N}_\vartheta(f,X))\geq\frac{1}{k}\log \lfloor\frac{1}{2}m\rfloor=\frac{1}{k}\log m\frac{\log\lfloor\frac{1}{2}m\rfloor}{\log m}>(\eta+\tau)\frac{\eta}{\eta+\tau}=\eta.$$
	Finally, by the arbitrariness of $0<\eta<h_{\mathrm{top}}(f,X)$, we have that $$h_{\mathrm{top}}(f,\mathcal{N}_\vartheta(f,X))=h_{\mathrm{top}}(f,X)=\frac{1}{L}h_{\mathrm{top}}(f^L,Y)>0.$$
\end{proof}

\section{Proof of Theorem \ref{Theorem C}}\label{Section 6}
In this section, we prove Theorem \ref{Theorem C} following the same line of the proof of Theorem \ref{Theorem B}.  It is enough to prove the following theorem.
\begin{theorem}\label{Theorem 5.1}
	Suppose that $(X,f)$ is a dynamical system and $L\in\mathbb{N}^+$, if there exists a subsystem $(Y,f^L)$ of $(X,f^L)$ such that
	\begin{enumerate}[(1)]
		\item $(Y,f^L)$ satisfies the modified almost specification property with $h_{\mathrm{top}}(f^L,Y)>0$;
		\item $X=\bigcup_{i=0}^{L-1}f^i(Y)$.
	\end{enumerate}
	Then $h_{\mathrm{top}}(f,X)=\frac{1}{L}h_{\mathrm{top}}(f^L,Y)>0$ and for any $0<\eta<h_{\mathrm{top}}(f,X)$ and any $N\in\mathbb{N}^+$, there exist $m,k\in\mathbb{N}^+$ with $$k\in\{Ln:n\in\mathbb{N}^+\}, k\geq N \text{ and } \frac{1}{k}\log m>\eta,$$ an alphabet $\mathcal{A}$ with $|\mathcal{A}|=m$ and a subsystem $(\Lambda,f^k)$ of $(Y,f^k)$ such that 
	\begin{enumerate}[(1)]
		\item $f^i(\Lambda)\cap f^j(\Lambda)=\emptyset$ for any $0\leq i<j\leq k-1$;
		\item there is a factor $\pi: (\Lambda,   f^k)\to (\mathcal{A}^{\mathbb{N}},  \sigma)$.
	\end{enumerate}
\end{theorem}
Before we give the proof, we prove some basic facts for the modified almost specification property. Recall that the \emph{measure center} of a dynamical system $(X,f)$ is  $$C_f(X):=\overline{\bigcup_{\mu\in\mathcal{M}_f(X)}S_{\mu}}.$$
Note that $f(S_{\mu})=S_{\mu}$ for any invariant measure $\mu$. Indeed, since $S_{\mu}$ is a compact set, $f(S_\mu)$ is compact and thus closed, and one has $f (S_\mu) \subseteq S_\mu \subseteq f^{-1}\left(f\left(S_\mu\right)\right)$, then 
$1=\mu(S_{\mu})\leq \mu(f^{-1}(f (S_{\mu})))=\mu(f (S_{\mu}))$,
hence we get $S_{\mu} \subseteq f (S_{\mu})$. As a result, $ \bigcup_{\mu \in \mathcal{M}_f(X)} S_\mu \subseteq f\left(C_f(X)\right)$, then we have  $ C_f(X) \subseteq f\left(C_f(X)\right)$. Therefore, $$ f\left(C_f(X)\right)=C_f(X) \text{ and thus }C_f(X)\subset \bigcap_{i=0}^{\infty}f^i(X).$$

Combining Definition \ref{Definition 2.3} with the compactness of $X$, we have the following. 
\begin{lemma}\label{Lemma 5.4}
	Suppose that $(X,f)$ is a dynamical system satisfying the modified almost specification property with the gap function $g$ and a function $k_g$. Then for any $\varepsilon>0$, any points $\{x_i\}_{i=1}^\infty\in C_f(X)$ and any integer $n\geq k_g(\varepsilon)$, we have $$\bigcap_{j=1}^\infty f^{-(j-1)n}(\overline{B_{n}(f;g(n,\varepsilon);x_j,\varepsilon)})\neq\emptyset.$$
\end{lemma}
Now, we give the proof of Theorem \ref{Theorem 5.1}.
\begin{proof}[Proof of Theorem \ref{Theorem 5.1}]
	Suppose that $(Y,f^L)$ satisfies the modified almost specification property with mistake function $g$ and a function $k_g$. We define a mistake function $g_\star$ by  $$g_\star(n,\varepsilon):=3Lg(n,\frac{1}{k_\varepsilon}\varepsilon)\text{ for any }n\in\mathbb{N}^+\text{ and any }\varepsilon>0,$$ where $$k_\varepsilon:=\min\{k\in\mathbb{N}^+:k\geq3,~d_L(f;x,y)\leq\frac{1}{3}\varepsilon\text{ for any }x,y\in X\text{ with }d(x,y)\leq\frac{1}{k}\varepsilon\}.$$
	
	By Lemma \ref{Lemma 2.2}, we have that $h_{\mathrm{top}}(f,X)=\frac{1}{L}h_{\mathrm{top}}(f^L,Y)>0$. Given $0<\eta<h_{\mathrm{top}}(f,X)$, we choose $\tau>0$ such that $$\eta L+4\tau<h_{\mathrm{top}}(f^L,Y).$$ By the variational principle, there exists $\mu\in\mathcal{M}_{f^L}^e(Y)$ such that $$h_\mu(f^L)>\eta L+3\tau.$$ By Lemma \ref{Lemma 2.4}, there exist $\delta_1>0$, $\varepsilon_1>0$ and $N_1\in\mathbb{N}^+$ such that for any $n\geq N_1$,   there exists a $(f^L;\delta_1, n,  \varepsilon_1)$-separated set $\Gamma_n(f^L)\subset S_\mu$ with $$\frac{\log|\Gamma_n(f^L)|}{n}\geq \eta L+3\tau.$$
	
	We define $$\nu:=\frac{1}{L}\sum_{i=0}^{L-1}\mu\circ f^{-i},$$ then  $\nu\in\mathcal{M}_f^e(X)$ and $h_\nu(f)>0$.
	
	We set $\delta_\star=\frac{1}{2L}$. Since $\nu(S_\mu)\geq\frac{1}{L}>\delta_\star$, by using Lemma \ref{Lemma C-New} for $g_\star$, there exists $\varepsilon_\star>0$ such that for any $0<\tilde{\varepsilon}\leq\varepsilon_\star$, any $N_2\in\mathbb{N}^+$ and any $\varepsilon_2>0$, there exist $l\geq N_2$, $F\subset S_\mu$ with $\nu(F)>0$ and $\operatorname{diam}(F)<\varepsilon_2$, such that
	\begin{equation}\label{equ (5.1)}
		\nu(\{x\in F:f^{il}x\in F\text{ and }R_{l}(f;g_\star(l,\tilde{\varepsilon});f^{il}x,\tilde{\varepsilon})>6Ll\text{ for any }0\leq i\leq 6L\})>0.
	\end{equation}
	
	We set $$\hat{\varepsilon}:=\min\{\frac{1}{6}\varepsilon_\star,\frac{1}{6}\varepsilon_1\},~\tilde{\varepsilon}:=3\hat{\varepsilon}\text{ and }\varepsilon:=\frac{1}{k_{\tilde{\varepsilon}}}\tilde{\varepsilon}\leq\frac{1}{3}\tilde{\varepsilon}=\hat{\varepsilon}.$$
	
	Then $g_\star(n,\tilde{\varepsilon})=3Lg(n,\varepsilon)$ and for any $x,y\in X$ with $d(x,y)\leq\varepsilon$, we always have 
	$$d_L(f;x,y)\leq\frac{1}{3}\tilde{\varepsilon}=\hat{\varepsilon}.$$
	
	Since $$\lim_{n\to\infty}\frac{g(n,\varepsilon)}{n}=0,$$ there exists $N_3\in\mathbb{N}$ such that for any $n\geq N_3$, we have
	
	\begin{equation}\label{equ (5.2)}
		6Lg(n,\varepsilon)<\delta_1 n.
	\end{equation}
	We choose $N_4\in\mathbb{N}^+$ large enough such that for any $n\geq N_4$, we have  
	\begin{equation}\label{equ (5.4)}
		e^{n(\eta L+3\tau)}>e^{n(\eta L+2\tau)}+6Ln.
	\end{equation}
	Now, we set $$N_2:=\max\{N,N_1,N_3,N_4,k_g(\varepsilon)\}\text{ and }\varepsilon_2:=\varepsilon,$$ get $l\geq N_2$, $F\subset S_\mu$ with $\nu(F)>0$ and $\operatorname{diam}(F)<\varepsilon$ satisfying equation (\ref{equ (5.1)}). In particular, we can choose $x_0\in F$ such that
	\begin{equation}\label{equ (5.5)}
		R_{l}(f;3Lg(l,\varepsilon);f^{il}x_0,3\hat{\varepsilon})>6Ll\text{ for any }0\leq i\leq 6L.
	\end{equation}
	
	\textbf{Claim 1.} For any $1\leq j\leq (6L-1)l$, we always have
	$$f^jx_0\notin \overline{B_l(f;3Lg(l,\varepsilon);x_0,2\hat{\varepsilon})}.$$
    \begin{proof}[Proof of Claim 1]
    	On the contrary, we assume that $$f^jx_0\in \overline{B_l(f;3Lg(l,\varepsilon);x_0,2\hat{\varepsilon})}.$$
    Then we have that $R_{l}(f;3Lg(l,\varepsilon);x_0,3\hat{\varepsilon})\leq (6L-1)l$, which is a contradiction to equation (\ref{equ (5.5)}).
    \end{proof}

	Since $l\geq N_1$, we can choose a $(f^L;\delta_1,l,\varepsilon_1)$-separated set $\Gamma_l(f^L)\subset S_\mu$ with $$\frac{\log|\Gamma_l(f^L)|}{l}\geq \eta L+3\tau.$$ Since $l\geq N_4$, by equation (\ref{equ (5.4)}), we have that 
	\begin{equation}\label{equ (5.6)}
		\frac{1}{l}\log(|\Gamma_{l}(f^L)|-6Ll)>\eta L+2\tau.
	\end{equation}
	Denote
	$$G:=\{x\in\Gamma_{l}(f^L): x\in\overline{B_{Ll}(f;3Lg(l,\varepsilon);f^jx_0,2\hat{\varepsilon})}\text{ for some } 0\leq j\leq 5Ll\}.$$
	
	\textbf{Claim 2.} We have $$|G|\leq 5Ll+1\leq6Ll.$$
	\begin{proof}[Proof of Claim 2]
		On the contrary, there exist $0\leq j\leq 5Ll$ and $$x\neq y\in\Gamma_{l}(f^L)\cap\overline{B_{Ll}(f;3Lg(l,\varepsilon);f^jx_0,2\hat{\varepsilon})}.$$
		Hence, $x\in\overline{B_{Ll}(f;6Lg(l,\varepsilon);y,4\hat{\varepsilon})}$. As a result, by equation (\ref{equ (5.2)}), we have $$|\{0\leq i\leq Ll-1:d(f^ix,f^iy)\leq 4\hat{\varepsilon}<\varepsilon_1\}|\geq Ll-6Lg(l,\varepsilon)>Ll-\delta_1l.$$
		Hence, $x$ and $y$ are not $(f^L;\delta_1,l,\varepsilon_1)$-separated, which is a contradiction.
	\end{proof}
	Denote $$\Gamma'_{l}(f^L):=\Gamma_{l}(f^L)\setminus G,$$ then combining Claim 2 with equation (\ref{equ (5.6)}), we have 
	\begin{equation}\label{equ (5.7)}
		\frac{1}{l}\log(|\Gamma'_{l}(f^L)|)>\eta L+2\tau.
	\end{equation}
	And for any $x\in \Gamma'_{l}(f^L)$ and any $0\leq j\leq 5Ll$, we always have that
	\begin{equation}\label{equ (5.8)}
		x\notin\overline{B_{Ll}(f;3Lg(l,\varepsilon);f^jx_0,2\hat{\varepsilon})}.
	\end{equation}
	
	We choose $p\in\mathbb{N}^+$ large enough such that 
	\begin{equation}\label{equ (5.9)}
		\frac{pl(\eta L+2\tau)}{6l+pl}>\eta L+\tau
	\end{equation} and denote $$k=L(6l+pl)\geq N\text{ and }m=|\Gamma'_{l}(f^L)|^p.$$ Then combining equation (\ref{equ (5.7)}) and equation (\ref{equ (5.9)}), we have that $$\frac{1}{k}\log m>\eta+\frac{1}{L}\tau>\eta.$$
	Now we denote an alphabet $$\mathcal{A}:=\Gamma'_{l}(f^L)^p=\Gamma'_{l}(f^L)\times\Gamma'_{l}(f^L)\times\cdots\times\Gamma'_{l}(f^L)\text{ with }|\mathcal{A}|=m.$$ 
	For a element $\theta=(\theta^1,\theta^2,\cdots,\theta^p)\in\mathcal{A}$, we define $$O(\theta):=\bigcap_{i=0}^{5}f^{-iL}(\overline{B_{Ll}(f;Lg(l,\varepsilon);f^{iL}x_0,\hat{\varepsilon})})\cap \bigcap_{i=1}^{p}f^{-6Ll-(i-1)Ll}(\overline{B_{Ll}(f;Lg(l,\varepsilon);\theta^i,\hat{\varepsilon})}).$$
	Recall that for any $x,y\in X$ with $d(x,y)\leq\varepsilon$, we always have 
	$d_L(f;x,y)\leq\hat{\varepsilon}$ and thus $$\overline{B_l(f^L;g(l,\varepsilon);x,\varepsilon)}\subset\overline{B_{Ll}(f;Lg(l,\varepsilon);x,\hat{\varepsilon})}$$
	Combining with the facts that $(Y,f^L)$ satisfies the modified almost specification property with mistake function $g$ and a function $k_g$, and $l\geq k_g(\varepsilon)$, we conclude that $Y\cap O(\theta)$ is non-empty and compact. 
	
	Given $\underline{\theta}=\theta_0\theta_1\cdots\in\mathcal{A}^{\mathbb{N}}$, we define $$Z_{\underline{\theta}}:=Y\cap\bigcap_{i\geq0}f^{-ki}(O(\theta_i)).$$
	Similarly, by Lemma \ref{Lemma 5.4}, we conclude that  $Z_{\underline{\theta}}$ is nonempty and closed.  We define $$\Lambda:=\bigcup_{\underline{\theta}\in\mathcal{A}^\mathbb{N}}Z_{\underline{\theta}},$$ then we have $f^k(\Lambda)\subset \Lambda\subset Y$.
	
	\textbf{Claim 3.} We claim that $Z_{\underline{\theta}}\cap Z_{\underline{\hat{\theta}}}=\emptyset$ for any $\underline{\theta}\neq\underline{\hat{\theta}}\in\mathcal{A}^\mathbb{N}$.
	\begin{proof}[Proof of Claim 3]
		Since $\underline{\theta}\neq\underline{\hat{\theta}}$, there exist $i\in\mathbb{N}$ and $1\leq j\leq p$ such that $\theta_i^{j}\neq\hat{\theta}_i^j\in\Gamma'_{l}(f^L)$. Denote $$r=ik+6Ll+(j-1)Ll,$$ then for any $z\in Z_{\underline{\theta}}$ and $\hat{z}\in Z_{\underline{\hat{\theta}}}$, we have $$f^rz\in \overline{B_{Ll}(f;Lg(l,\varepsilon);\theta_i^{j},\hat{\varepsilon})}\text{ and }f^r\hat{z}\in \overline{B_{Ll}(f;Lg(l,\varepsilon);\hat{\theta}_i^{j},\hat{\varepsilon})}.$$ Since $\hat{\varepsilon}\leq\frac{1}{6}\varepsilon_1$ and $\theta_i^{j}\neq\hat{\theta}_i^j\in\Gamma'_{l}(f^L)$, we have  $$d(f^{r+t}z,f^{r+t}\hat{z})>\varepsilon_1-2\hat{\varepsilon}\geq\frac{2}{3}\varepsilon_1>0\text{ for some }0\leq t\leq Ll-1.$$ Hence, $z\neq\hat{z}$, as a result, $Z_{\underline{\theta}}\cap Z_{\underline{\hat{\theta}}}=\emptyset$.
	\end{proof}

	\textbf{Claim 4.} We claim that $\Lambda$ is closed and hence $(\Lambda,f^k)$ is a subsystem of $(X,f^k)$.
	\begin{proof}[Proof of Claim 4]
		Since $f^k(\Lambda)\subset \Lambda$, it suffices to show that $\Lambda$ is closed. 	Given $\underline{\theta}=\theta_0\theta_1\cdots\in\mathcal{A}^{\mathbb{N}}$ and $r\in\mathbb{N}$, we define $$Z_{\underline{\theta},r}:=Y\cap\bigcap_{i=0}^{r}f^{-ki}(O(\theta_i)).$$ Then $Z_{\underline{\theta},r}$ is nonempty, closed and for any $\underline{\theta},\underline{\hat{\theta}}\in\mathcal{A}^\mathbb{N}$ with $\theta_0\theta_1\cdots\theta_r=\hat{\theta}_0\hat{\theta}_1\cdots\hat{\theta}_r$, we always have that $Z_{\underline{\theta},r}=Z_{\underline{\hat{\theta}},r}$. Therefore, $$\Lambda=\bigcup_{\underline{\theta}\in\mathcal{A}^\mathbb{N}}Z_{\underline{\theta}}=\bigcap_{r\geq 0}\bigcup_{\underline{\theta}\in\mathcal{A}^\mathbb{N}}Z_{\underline{\theta},r}$$ is closed.
	\end{proof}
	
	Now, we define a map $\pi: (\Lambda,   f^k)\to (\mathcal{A}^{\mathbb{N}},  \sigma)$ as $\pi(z):=\underline{\theta}$ for any $z\in Z_{\underline{\theta}}$. Then $\pi\circ f^k=\sigma\circ\pi$. By Claim 3, $\pi$ is well-defined. Since $Z_{\underline{\theta}}\neq\emptyset$ for any $\underline{\theta}\in\mathcal{A}^\mathbb{N}$, we conclude that $\pi$ is surjective. 
	
	\textbf{Claim 5.} We claim that $\pi$ is continuous.
	\begin{proof}[Proof of Claim 5]
		It is enough to show that for any $r\in\mathbb{N}$, any $z\in Z_{\underline{\theta}}$ and $\hat{z}\in Z_{\underline{\hat{\theta}}}$ with $d(f^iz,f^i\hat{z})<\frac{2}{3}\varepsilon_1$ for any $0\leq i\leq (r+1)k-1$, we always have that $\theta_0\theta_1\cdots\theta_r=\hat{\theta}_0\hat{\theta}_1\cdots\hat{\theta}_r$. We assume that there exist $0\leq s\leq r$ and $1\leq j\leq p$ such that $\theta_s^j\neq \hat{\theta}_s^j\in\Gamma'_{l}(f^L)$, then follow the same line of the proof of Claim 3, we conclude that $$d(f^tz,f^t\hat{z})>\frac{2}{3}\varepsilon_1\text{ for some }0\leq t\leq (r+1)k-1.$$ It is a contradiction to $d(f^iz,f^i\hat{z})<\frac{2}{3}\varepsilon_1$ for any $0\leq i\leq (r+1)k-1$.
	\end{proof}
	As a result, $\pi$ is factor. Finally, we show that $f^i(\Lambda)\cap f^j(\Lambda)=\emptyset$ for any $0\leq i<j\leq k-1$. When $k=1$, the conclusion is clear. Hence, we may assume in the following that $k\geq2$. It is enough to show that $\Lambda\cap f^r(\Lambda)=\emptyset$ for any $1\leq r\leq k-1$.
	
	Suppose that $\Lambda\cap f^{r}(\Lambda)\neq\emptyset$ for some $1\leq r\leq k-1$. Then for any $z\in\Lambda\cap f^{r}(\Lambda)$, there exist $\underline{\theta},\underline{\hat{\theta}}\in\mathcal{A}^\mathbb{N}$ such that $z\in Z_{\underline{\theta}}\cap f^{r}(Z_{\underline{\hat{\theta}}})$. 
	
	If $1\leq r\leq (6L-1)l$, then we have $$z\in\overline{B_{Ll}(f;Lg(l,\varepsilon);x_0,\hat{\varepsilon})}\cap\overline{B_{l}(f;2Lg(l,\varepsilon);f^rx_0,\hat{\varepsilon})}.$$ Hence, $$f^rx_0\in\overline{B_{l}(f;3Lg(l,\varepsilon);x_0,2\hat{\varepsilon})}.$$ It is a contradiction to Claim 1.
	
	If $(6L-1)l< r< k-(6L-1)l$, choose $y\in\Lambda$ with $f^ry=z$, then there always exist $l-1\leq i\leq l-1+Ll-1$ and $1\leq t\leq p$ such that $$f^iz=f^{r+i}y\in\overline{B_{Ll}(f;Lg(l,\varepsilon);\theta_0^t,\hat{\varepsilon})}
	\text{ and }f^{i}z\in\overline{B_{Ll}(f;2Lg(l,\varepsilon);f^ix_0,\hat{\varepsilon})}$$ As a result, $$\theta_0^t\in\overline{B_{Ll}(f;3Lg(l,\varepsilon);f^ix_0,2\hat{\varepsilon})}.$$ It is a contradiction to equation (\ref{equ (5.8)}).
	
	If $k-(6L-1)l\leq r\leq k-1$, then $1\leq k-r\leq(6L-1)l$. Since there exists $y\in\Lambda$ with $f^ry=z$, we have $$f^{k-r}z\in\overline{B_{Ll}(f;Lg(l,\varepsilon);x_0,\hat{\varepsilon})}\cap\overline{B_{l}(f;2Lg(l,\varepsilon);f^{k-r}x_0,\hat{\varepsilon})}.$$ Hence, $$f^{k-r}x_0\in\overline{B_{l}(f;3Lg(l,\varepsilon);x_0,2\hat{\varepsilon})}.$$ It is a contradiction to Claim 1.
	
	In conclusion, we have that $\Lambda\cap f^r(\Lambda)=\emptyset$ for any $1\leq r\leq k-1$.
	
	Now, we have finished the proof of Theorem \ref{Theorem 5.1}.
	 
\end{proof}

\section{A detailed explanation for Remark \ref{Remark 1.1} and Remark \ref{Remark 1.2}}\label{Section 7} 
In this section, we give a detailed explanation for Remark \ref{Remark 1.1} and Remark \ref{Remark 1.2}. First, we prove the following theorem.
\begin{theorem}\label{Theorem 6.1}
	Suppose that $(X,f)$ is a dynamical system and $L\in\mathbb{N}^+$, if there exists a subsystem $(Y,f^L)$ of $(X,f^L)$ such that
	\begin{enumerate}[(1)]
		\item $(Y,f^L)$ satisfies the shadowing property;
		\item $X=\bigcup_{i=0}^{L-1}f^i(Y)$;
		\item $f^{i}(Y)\cap f^{j}(Y)=\emptyset$ for any $0\leq i<j\leq L-1$.
	\end{enumerate}
	Then $(X,f)$ satisfies the shadowing property.
\end{theorem}
\begin{proof}
	Given $\varepsilon>0$, we choose $0<\varepsilon_1\leq\varepsilon$ such that for any $x,y\in X$ with $d(x,y)<\varepsilon_1$, we always have that $d_{L}(f;x,y)<\frac{1}{2}\varepsilon$. Since $(Y,f^L)$ satisfies the shadowing property, there exists $\delta_1>0$ such that every $\delta_1$-pseudo orbit of $f^L$ is $\varepsilon_1$-traced by some point in $Y$ for $f^L$. Since $f^{i}(Y)\cap f^{j}(Y)=\emptyset$ for any $0\leq i<j\leq L-1$, we can choose $\delta_2>0$ such that for any $x,y\in X$ with $d(x,y)<\delta_2$, we always have that $x,y\in f^i(Y)\text{ for some }0\leq i\leq L-1$. Now, we  choose $$0<\delta\leq\frac{1}{L}\min\{\delta_1,\delta_2,\frac{1}{2}\varepsilon\}$$ such that for any $x,y\in X$ with $d(x,y)<\delta$, we always have that $$d_{L}(f;x,y)<\frac{1}{L}\min\{\delta_1,\delta_2,\frac{1}{2}\varepsilon\}.$$ We suppose that $\mathcal{L}:=x_0x_1\cdots$ is a $\delta$-pseudo orbit of $f$ with $x_0\in f^k(Y)$ for some $0\leq k\leq L-1$. We choose $y_0\in Y$ with $f^ky=x_0$ and define 
	\[y_i=\begin{cases}
		f^iy_0 & \text{if }0\leq i\leq k-1,\\
		x_{i-k} & \text{if }i\geq k.
	\end{cases}
	\]
	Then $\widetilde{\mathcal{L}}:=y_0y_1\cdots$ is a $\delta$-pseudo orbit of $f$ with $y_0\in Y$. For any $r\in\mathbb{N}$ and any $1\leq s\leq L$, we have 
	\begin{equation*}
		\begin{split}
			d(f^{s}y_{rL},y_{rL+s})&\leq d(f^{s}y_{rL},fy_{rL+s-1})+d(fy_{rL+s-1},y_{rL+s})\\
			&\leq d(f^{s}y_{rL},f^2y_{rL+s-2})+d(f^2y_{rL+s-2},fy_{rL+s-1})\\&\quad+d(fy_{rL+s-1},y_{rL+s})\\
			&\leq\sum_{i=0}^{s-1}d(f^{i+1}y_{rL+s-(i+1)},f^iy_{rL+s-i})\\&< s\frac{1}{L}\min\{\delta_1,\delta_2,\frac{1}{2}\varepsilon\}\\&\leq\min\{\delta_1,\delta_2,\frac{1}{2}\varepsilon\}.
		\end{split}
	\end{equation*}
	Hence, we always have that $y_{rL}\in Y$ for any $r\in\mathbb{N}$ and $\widetilde{\mathcal{L}}_L:=y_0y_Ly_{2L}\cdots$ is a $\delta_1$-pseudo orbit of $f^L$. Hence, we can find $x\in Y$ such that
	\begin{equation}\label{equ (6.1)}
		d(f^{iL}x,y_{iL})<\varepsilon_1\leq\varepsilon\text{ for any }i\in\mathbb{N}.
	\end{equation} As a result, for any $r\in\mathbb{N}$ and any $1\leq s\leq L-1$, we have 
	\begin{equation}\label{equ (6.2)}
		d(f^{rL+s}x,y_{rL+s})\leq d(f^{rL+s}x,f^sy_{rL})+d(f^sy_{rL},y_{rL+s})<\frac{1}{2}\varepsilon+\frac{1}{2}\varepsilon=\varepsilon.
	\end{equation} 
	Combining equation (\ref{equ (6.1)}) and equation (\ref{equ (6.2)}), we have that $$d(f^ix,y_i)<\varepsilon\text{ for any }i\in\mathbb{N}.$$ From the definition of $\{y_i\}_{i\in\mathbb{N}}$, we know that $$d(f^i(f^kx),x_i)<\varepsilon\text{ for any }i\in\mathbb{N}.$$ Therefore, $(X,f)$ satisfies the shadowing property.
\end{proof}
If we remove the assumption $(3)$ in Theorem \ref{Theorem 6.1}, then we have a counterexample. More precisely, we have  
\begin{theorem}\label{Theorem 6.2}
	There exist a dynamical system $(X,f)$, $L\in\mathbb{N}^+$, and a subsystem $(Y,f^L)$ of $(X,f^L)$ such that
	\begin{enumerate}[(1)]
		\item $(Y,f^L)$ satisfies the shadowing property and the specification property;
		\item $X=\bigcup_{i=0}^{L-1}f^i(Y)$.
	\end{enumerate}
	However, $(X,f)$ does not satisfy the shadowing property or the modified almost specification property.
\end{theorem}
We will construct the counterexample in a two-sided full shift. We first recall some basic definitions and facts for two-sided full shifts. Let $\mathcal{A}$ be a finite alphabet with $|\mathcal{A}|\geq2$ and let $Y$ be a non-empty compact subset of $\mathcal{A}^{\mathbb{Z}}$, then $(Y,\sigma)$ is said to be a \emph{(two-sided) subshift} of $(\mathcal{A}^{\mathbb{Z}},\sigma)$ if $\sigma(Y)=Y$. Given a subset $\mathcal{F}\subset \mathcal{A}^\ast$, we can define the set $$X=X(\mathcal{F}):=\{x\in \mathcal{A}^{\mathbb{Z}}:x_ix_{i+1}\cdots x_j\notin\mathcal{F},\text{for any integers }i\leq j\},$$ then $(X,\sigma)$ is a subshift of $(\mathcal{A}^{\mathbb{Z}},\sigma)$ and $\mathcal{F}$ is called the \emph{forbidden words} of $(X,\sigma)$. We say that $(X,\sigma)$ is a \emph{(two-sided) subshift of finite type} if $|\mathcal{F}|<\infty$, if further $|\mathcal{F}|\leq n$ for some $n\in\mathbb{N}$, then $(X,\sigma)$ is said to be a subshift of finite type of order $n$. A subshift satisfies the shadowing property if and only if it is of finite type \cite[Theorem 1]{Walters2}. 
\begin{lemma}\cite[Proposition 17.6]{DGS}\label{Lemma 6.3}
	Suppose that $(X,\sigma)$ is a subshift of $(\mathcal{A}^{\mathbb{Z}},\sigma)$. Then it is of finite type of order $n$ for some $n\in\mathbb{N}^+$ if and only if for every pair $x,y\in X$ and every $k\in\mathbb{Z}$ such that $x_i=y_i$ for any $k\leq i\leq k+n-2$, the point $z\in\mathcal{A}^{\mathbb{Z}}$ defined by 
	\[z_i:=
	\begin{cases}
		x_i & \text{if }i\leq k+n-2,\\
		y_i & \text{if }i\geq k.
	\end{cases}\] belongs to $X$.
\end{lemma}
Now, we give the proof of Theorem \ref{Theorem 6.2}.
\begin{proof}[Proof of Theorem \ref{Theorem 6.2}]
	We consider the two-sided full shift $(\{0,1,2,3\}^\mathbb{Z},\sigma)$. We define $$\omega_1:=00, \omega_2:=01, \omega_3:=10, \omega_4:=02$$ and an alphabet $$\mathcal{A}:=\{\omega_1,\omega_2,\omega_3,\omega_4\}.$$ Then we can define a compact subset $\mathcal{E}$ of $\{0,1,2\}^\mathbb{Z}$ as $$\mathcal{E}:=\{a=\cdots a_{-1}a_0a_1\cdots:a_i\in\mathcal{A}\text{ for any }i\in\mathbb{Z}\}.$$ As a result, we have that $\sigma^2(a)_i=a_{i+1}$ for any $a\in\mathcal{E}$ and $i\in\mathbb{Z}$. It is clear that $$(\mathcal{E},\sigma^2)\text{ is topologically conjugate to }  (\{0,1,2,3\}^\mathbb{Z},\sigma)$$ and thus satisfies the shadowing property and the specification property.
	We define $$X:=\mathcal{E}\cup\sigma(\mathcal{E}).$$ Then $(X,\sigma)$ is a subshift of $(\{0,1,2\}^\mathbb{Z},\sigma)$. For every $x\in X$, all occurrences of the symbol $2$ in $x$
	have the same parity. Indeed, if $x\in \mathcal{E}$, then each block
	\[
	x_{2k}x_{2k+1}\in \{00,01,10,02\},
	\]
	so $2$ can appear only at odd coordinates; similarly, if $x\in \sigma(\mathcal{E})$, then $2$
	can appear only at even coordinates.

	Given $m\in\mathbb{N}^+$, we define $x,y\in X$ as follows:
	$$x:=(02)^\infty10(01)^m00(01)^\infty\text{ with }x_0x_1\cdots:=(01)^m00(01)^\infty$$ and $$y:=(02)^\infty00(10)^m(02)^\infty\text{ with }y_0y_1\cdots:=0(10)^m(02)^\infty=(01)^m00(20)^\infty.$$ Then $$x_i=y_i\text{ for any }0\leq i\leq (2m+3)-2.$$ We define $z\in\{0,1,2\}^\mathbb{Z}$ as follows:
	\[z_i:=
	\begin{cases}
		x_i & \text{if }i\leq(2m+3)-2,\\
		y_i & \text{if }i\geq 0.
	\end{cases}
	\]
	Then $$z=(02)^\infty10(01)^m00(20)^\infty\text{ with }z_0z_1\cdots=(01)^m00(20)^\infty.$$ As a result, $$z_{-3}=z_{2m+2}=2\text{ and }z_i\neq 2\text{ for any }-3<i<2m+2.$$ Since $2m+2-(-3)=2(m+2)+1$ is an odd number, we conclude that $z\notin X$. By Lemma \ref{Lemma 6.3}, we have that $(X,\sigma)$ is not a subshift of finite type of order $(2m+3)$. By the arbitrariness of $m\in\mathbb{N}^+$, we conclude that $(X,\sigma)$ is not a subshift of finite type. As a result, $(X,f):=(X,\sigma)$ does not satisfy the shadowing property.
	
	Finally, we show that $(X,\sigma)$ does not satisfy the modified almost
	specification property. Let
	\[
	\alpha:=(02)^\infty \in \mathcal{E} \subset X\text{ with }\alpha_0=0\text{ and }
	\beta:=(20)^\infty=\sigma(\alpha)\in \sigma(\mathcal{E})\subset X.
	\]
	
	Suppose that $(X,\sigma)$ satisfies the modified almost
	specification property with mistake function $g$ and a function $k_g$. Fix $0<\varepsilon<1$. Since
	$g(n,\varepsilon)/n\to 0$ as $n\to\infty$, we can choose $n$ large enough such that
	\[
	2n\ge k_g(\varepsilon)\text{ and } g(2n,\varepsilon)<n.
	\]
	By Definition \ref{Definition 2.3}, there exists
	$z\in X$ such that
	\[
	z\in B_{2n}(\sigma;g(2n,\varepsilon);\alpha,\varepsilon)
	\]
	and
	\[
	\sigma^{2n}z\in B_{2n}(\sigma;g(2n,\varepsilon);\beta,\varepsilon).
	\]
	Since $\varepsilon<1$, for any $z',z''\in X$ with $d(z',z'')<\varepsilon$, we have that $z'_0=z''_0$. As a result, we conclude that the Hamming distance between
	$z_0z_1\cdots z_{2n-1}$ and
	\[
	a_n:=(02)^n
	\]
	is at most $g(2n,\varepsilon)$, and the Hamming distance between
	$z_{2n}z_{2n+1}\cdots z_{4n-1}$ and
	\[
	b_n:=(20)^n
	\]
	is also at most $g(2n,\varepsilon)$.
	
	However, in the concatenated word $a_nb_n$ of length $4n$, the symbol $2$ appears in
	$a_n$ exactly at the odd coordinates
	\[
	1,3,\dots,2n-1,
	\]
	and in the $b_n$-part exactly at the even coordinates
	\[
	2n,2n+2,\dots,4n-2.
	\] Since every admissible word of $X$ can contain
	the symbol $2$ only on one parity class, any admissible word of length $4n$ must
	differ from $a_n$ in at least $n$ coordinates or differ from $b_n$ in at least $n$
	coordinates. This contradicts the fact that both Hamming distances are at most
	$g(2n,\varepsilon)<n$. Therefore, $(X,\sigma)$ does not satisfy the modified almost
	specification property.
\end{proof}

\section{Applications}\label{Section 8}
In this section, we present further applications of our main results and prove Corollary \ref{Corollary B}.
\subsection{Dynamical systems having a semi-horseshoe}
Note that $(\{0,1\}^{\mathbb{N}},\sigma)$ satisfies the shadowing property. Hence, we can use Theorem \ref{Theorem B} to strengthen Theorem \ref{Theorem A} as follows.
\begin{theorem}\label{new-theorem8.1}
	If $(X,f)$ has a semi-horseshoe, then for any non-trivial weight $\vartheta\in \ell^\infty(\mathbb{N})$, we have $h_{\mathrm{top}}(f,\mathcal{N}_\vartheta(f,X))>0$. 
\end{theorem}
\begin{proof}
	By Lemma \ref{new-lemma3.4}, we can always assume that $(X,f)$ has a one-sided $N$-order semi-horseshoe $(\Lambda,f^N)$ for some $N\in\mathbb{N}^+$. We define $M$ by equation \eqref{equ-3.1} and $\tau=NM$. Then by the proof of Theorem \ref{Theorem A}, there exists an admissible set $B_\ast\subset B_0$ such that
	\[
	Y_q(B_\ast)\subsetneq \{0,1\}^{\mathbb{N}}\text{ for any }q\in Q.
	\]
	By using Lemma \ref{new-lemma3.8} for $B_\ast$, there exists $L\in\mathbb{N}^+$ such that $(X,f)$ has a one-sided $\tau L$-order semi-horseshoe $(E,f^{\tau L})$ such that $E\subset B_\ast$ and $$E\cap f^i(E)=\emptyset\text{ for any }1\leq i\leq \tau L-1.$$ 
	
	Denote $k=\tau L$, we claim that $f^i(E)\cap f^j(E)=\emptyset$ for any $0\leq i<j\leq k-1$. Indeed, if $f^i(E)\cap f^j(E)\neq\emptyset$ for some $0\leq i<j\leq k-1$, denote $r=j-i$, then $$\emptyset\neq f^{k-i}(f^{i}(E)\cap f^{j}(E))\subset f^k(E)\cap f^{j+k-i}(E)\subset E\cap f^{j-i}(E)=E\cap f^r(E).$$

	Now, given a non-trivial weight $\vartheta\in \ell^\infty(\mathbb{N})$, since $f^i(E)\cap f^j(E)=\emptyset$ for any $0\leq i<j\leq k-1$, by Lemma \ref{Lemma 4.2}, there exist $0\leq s\leq k-1$ and a non-trivial weight $\varpi\in \ell^\infty(\mathbb{N})$ such that $$f^{k-s}(\mathcal{N}_\varpi(f^k,E))\subset\mathcal{N}_\vartheta(f,X).$$ Hence, by Lemma \ref{Lemma 2.2}, we have
	\begin{multline}\label{equ (8.15)}
		h_{\mathrm{top}}(f,\mathcal{N}_\vartheta(f,X))\geq h_{\mathrm{top}}(f,f^{k-s}(\mathcal{N}_\varpi(f^k,E)))\\=h_{\mathrm{top}}(f,\mathcal{N}_\varpi(f^k,E))=\frac{1}{k}h_{\mathrm{top}}(f^k,\mathcal{N}_\varpi(f^k,E)).
	\end{multline}
	By Lemma \ref{Lemma 4.3}, we have
	\begin{equation}\label{equ (8.16)}
		h_{\mathrm{top}}(f^k,\mathcal{N}_{\varpi}(f^k,E))\geq h_{\mathrm{top}}(\sigma,\mathcal{N}_{\varpi}(\sigma,\{0,1\}^{\mathbb{N}})).
	\end{equation}
	By Theorem \ref{Theorem A}, we have that
	\begin{equation}\label{equ (8.17)}
		h_{\mathrm{top}}(\sigma,\mathcal{N}_{\varpi}(\sigma,\{0,1\}^{\mathbb{N}}))=\log 2.
	\end{equation}
	Combining equations (\ref{equ (8.15)})-(\ref{equ (8.17)}), we have that $$h_{\mathrm{top}}(f,\mathcal{N}_\vartheta(f,X))\geq\frac{1}{k}\log 2>0.$$
\end{proof}
\subsection{Dynamical systems having a horseshoe}
Note that $(\{0,1\}^{\mathbb{Z}},\sigma)$ also satisfies the shadowing property. For dynamical systems having an $N$-order horseshoe, we can strengthen \cite[Theorem 1.1]{Fan-Fan-Ryzhikov-Shen2022} as follows.

\begin{theorem}
	Suppose that $(X,f)$ has an $N$-order horseshoe. Then for any non-trivial weight $\vartheta\in \ell^\infty(\mathbb{N})$, we have
	\[
	h_{\mathrm{top}}(f,\mathcal{N}_\vartheta(f,X))\ge \frac{1}{N}\log 2.
	\]
\end{theorem}

\subsection{Smooth dynamics}

In smooth dynamics, a \emph{hyperbolic horseshoe} is a topologically transitive, locally maximal hyperbolic set that is totally disconnected and infinite. Every hyperbolic horseshoe is expansive \cite[Corollary 6.4.10]{Katok-Hasseblatt-1995} and satisfies the shadowing property \cite[Theorem 18.1.2]{Katok-Hasseblatt-1995}. Hence, by an expansive version of \cite[Theorem 4.7]{Li-Oprocha-2013}, every smooth dynamical system having a hyperbolic horseshoe also has a horseshoe, and thus \cite[Theorem 1.1]{Fan-Fan-Ryzhikov-Shen2022} already implies Bohr chaoticity in this setting. Our results yield more precise information on the topological entropy of the set $\mathcal{N}_\vartheta(f,M)$.

Recall that an invariant measure is said to be \emph{hyperbolic} if its Lyapunov exponents are non-zero for almost every point. Let $M$ be a compact connected boundaryless smooth manifold with $\operatorname{dim}(M)\ge 1$, and let $f$ be a $C^{1+\alpha}$ diffeomorphism on $M$ with $\alpha>0$.

\begin{proof}[Proof of Corollary \ref{Corollary B}]
	Suppose that $\mu\in\mathcal{M}_f^e(M)$ is hyperbolic and satisfies $h_\mu(f)>0$. By Katok's classical approximation theorem for hyperbolic ergodic measures with positive metric entropy \cite{Katok-1980}; see also \cite[Theorem 3.3]{Avila-Crovisier-Wilkinson-2021}, we have
	\begin{equation}\label{equ (1.1)}
		h_\mu(f)\le \sup\bigl\{h_{\mathrm{top}}(f,\Lambda):(\Lambda,f)\text{ is a hyperbolic horseshoe of }(M,f)\bigr\}.
	\end{equation}
	
	If, in addition, $\operatorname{dim}(M)=2$ and $h_{\mathrm{top}}(f,M)>0$, then by Ruelle's inequality \cite{Ruelle-1978}, every $\mu\in\mathcal{M}_f^e(M)$ with $h_\mu(f)>0$ is hyperbolic. Hence
	\[
	h_{\mathrm{top}}(f,M)
	=
	\sup\bigl\{h_\mu(f):\mu\in\mathcal{M}_f^e(M)\text{ is hyperbolic and }h_\mu(f)>0\bigr\}.
	\]
	Therefore, Theorem \ref{Theorem B} implies Corollary \ref{Corollary B}.
\end{proof}

We now consider the $C^1$ case. Suppose that $f$ is a $C^1$ diffeomorphism on $M$ and $\Lambda$ is an $f$-invariant set. For two $Df$-invariant bundles $E,F\subset TM|_\Lambda$, we say that $F$ is \emph{dominated} by $E$ if there exist constants $C>0$ and $\lambda\in(0,1)$ such that for any $x\in\Lambda$ and any $n\in\mathbb{N}$,
\[
\|Df^n|_{F(x)}\|\cdot \|Df^{-n}|_{E(f^n x)}\|\le C\lambda^n.
\]
In this case, we write $F\oplus_\prec E$. A $f$-invariant set $\Lambda$ is said to admit a \emph{dominated splitting} if there exists a non-trivial $Df$-invariant splitting
\[
TM|_\Lambda=F\oplus_\prec E.
\]

\begin{definition}\label{Definition 2.9}
	Suppose that $\mu\in\mathcal{M}_f^e(M)$ is hyperbolic, has a positive Lyapunov exponent and a negative Lyapunov exponent. Let
	\[
	E_1\oplus\cdots\oplus E_l\oplus E_{l+1}\oplus\cdots\oplus E_{l+s}
	\]
	be its Oseledets splitting with corresponding Lyapunov exponents
	\[
	\lambda_1<\cdots<\lambda_l<0<\lambda_{l+1}<\cdots<\lambda_{l+s},
	\]
	defined for $\mu$-a.e. $x\in M$. Denote
	\[
	E^-=E_1\oplus\cdots\oplus E_l
	\qquad\text{and}\qquad
	E^+=E_{l+1}\oplus\cdots\oplus E_{l+s}.
	\]
	We say that $\mu$ \emph{admits a dominated splitting corresponding to the stable/unstable subspaces of its Oseledets splitting} if its support $S_\mu$ admits a dominated splitting $F\oplus_\prec E$ such that $F(x)=E^-(x)$ and $E(x)=E^+(x)$ for $\mu$-a.e. $x\in S_\mu$.
\end{definition}

Now suppose that $\mu\in\mathcal{M}_f^e(M)$ is hyperbolic with $h_\mu(f)>0$. By \cite[Theorem 3.17]{ST2015} or \cite[Theorem 1]{Gelfert-2016,Gelfert-2022}, if $\mu$ admits a dominated splitting corresponding to the stable/unstable subspaces of its Oseledets splitting, then $\mu$ can still be approximated by hyperbolic horseshoes, and hence the estimate \eqref{equ (1.1)} remains valid.

Let $\mathrm{Diff}_{\mathrm{vol}}^1(M)$ denote the space of $C^1$ volume-preserving diffeomorphisms on $M$, and let $m_\star$ be the normalized volume measure. Avila, Crovisier and Wilkinson \cite[Theorem B]{ACW} proved that there exists a residual subset $\mathcal{R}_\star$ of $\mathrm{Diff}_{\mathrm{vol}}^1(M)$ such that for any $f\in\mathcal{R}_\star$ with $h_{m_\star}(f)>0$, the map $f$ is non-uniformly Anosov and hence $m_\star$ admits a dominated splitting corresponding to the stable/unstable subspaces of its Oseledets splitting. Therefore, Theorem \ref{Theorem B} yields the following consequence.

\begin{theorem}
	Suppose that $f$ is a $C^1$ diffeomorphism on a compact connected boundaryless smooth manifold $M$, and that $\mu\in\mathcal{M}_f^e(M)$ is hyperbolic with $h_\mu(f)>0$ and admits a dominated splitting corresponding to the stable/unstable subspaces of its Oseledets splitting. Then for any non-trivial weight $\vartheta\in \ell^\infty(\mathbb{N})$, we have
	\[
	h_{\mathrm{top}}(f,\mathcal{N}_\vartheta(f,M))\ge h_\mu(f)>0.
	\]
	As a consequence, there exists a residual subset $\mathcal{R}_\star$ of $\mathrm{Diff}_{\mathrm{vol}}^1(M)$ such that for any $g\in\mathcal{R}_\star$ with $h_{m_\star}(g)>0$ and any non-trivial weight $\vartheta\in \ell^\infty(\mathbb{N})$,
	\[
	h_{\mathrm{top}}(g,\mathcal{N}_\vartheta(g,M))\ge h_{m_\star}(g)>0.
	\]
\end{theorem}

\subsection{$C^0$-generic conservative dynamics}

Let $M$ be a compact connected smooth Riemannian manifold with $\operatorname{dim}(M)\ge 2$. Let $\mu_*\in\mathcal{M}(M)$ be non-atomic (i.e. $\mu_*(\{x\})=0$ for any $x\in M$), zero on the boundary (i.e. $\mu_*(\partial M)=0$), and of full support (i.e. $S_{\mu_*}$ is $M$). We denote
\[
\mathcal{H}(M,\mu_*):=\{f\in\mathcal{H}(M):\mu_*\in\mathcal{M}_f(M)\}.
\]
Guih\'eneuf and Lefeuvre \cite[Corollary 1.4]{Guih\'eneuf-Lefeuvre-2018} showed that there exists a residual subset $\mathcal{R}'_*$ of $\mathcal{H}(M,\mu_*)$ such that every $f\in\mathcal{R}'_*$ satisfies the specification property. On the other hand, \cite[Theorem 3.17]{Guih\'eneuf-2012} implies that there exists a residual subset $\mathcal{R}''_*$ of $\mathcal{H}(M,\mu_*)$ such that every $f\in\mathcal{R}''_*$ has infinite topological entropy. Let
\[
\mathcal{R}_*=\mathcal{R}'_*\cap \mathcal{R}''_*.
\]
Then Theorem \ref{Theorem C} gives the following result.

\begin{theorem}
	Suppose that $M$ is a compact connected smooth manifold with $\operatorname{dim}(M)\ge 2$ and that $\mu_*\in\mathcal{M}(M)$ is non-atomic, zero on the boundary and of full support. Then there exists a residual subset $\mathcal{R}_*$ of $\mathcal{H}(M,\mu_*)$ such that for any $f\in\mathcal{R}_*$ and any non-trivial weight $\vartheta\in \ell^\infty(\mathbb{N})$,
	\[
	h_{\mathrm{top}}(f,\mathcal{N}_\vartheta(f,M))=h_{\mathrm{top}}(f,M)=\infty.
	\]
\end{theorem}

\subsection{Other applications}

In this subsection, we apply the main results to the following commonly studied classes of dynamical systems:
\begin{enumerate}[(1)]
	\item topologically transitive continuous self-maps on a compact interval or on the circle with positive topological entropy;
	\item subshifts with positive topological entropy that are factors of an $S$-gap shift or a $\beta$-shift.
\end{enumerate}

Let $n\in\mathbb{N}^+$. A collection $\mathcal{D}=\{D_0,D_1,\cdots,D_{n-1}\}$ is said to be a \emph{regular periodic decomposition} of $(X,f)$ if the following conditions hold:
\begin{enumerate}[(1)]
	\item $D_i=\overline{\operatorname{int}(D_i)}$ for every $0\le i\le n-1$, where $\operatorname{int}(D_i)$ denotes the interior of $D_i$;
	\item $D_i\cap D_j$ is nowhere dense for any $0\le i<j\le n-1$;
	\item $f(D_i)\subset D_{i+1 \,(\mathrm{mod}\, n)}$ for every $0\le i\le n-1$;
	\item $\bigcup_{i=0}^{n-1}D_i=X$.
\end{enumerate}

\begin{lemma}\cite[Lemma 2.1]{Banks1997}\label{Lemma 7.1}
	Suppose that $(X,f)$ is topologically transitive and has a regular periodic decomposition $\mathcal{D}=\{D_0,D_1,\cdots,D_{n-1}\}$. Then $D_0=f^n(D_0)$ and $D_j=f^j(D_0)$ for every $0\le j\le n-1$.
\end{lemma}

\begin{definition}\cite[Definition 37]{Kwietniak-Lacka-Oprocha-2016}
	A class $P$ of compact dynamical systems is called a \emph{property} if it is saturated with respect to topological conjugacy; that is, if $(X,f)\in P$ and $(Y,g)$ is topologically conjugate to $(X,f)$, then $(Y,g)\in P$. Let $P$ be a property of compact dynamical systems (e.g.\ topological transitivity, topological (weak) mixing, specification). A dynamical system $(X,f)$ is said to have \emph{property $P$ relative to a regular periodic decomposition} $\mathcal{D}=\{D_0,\cdots,D_{n-1}\}$ if $(D_i,f^n|_{D_i})$ has property $P$ for each $i\in\{0,\cdots,n-1\}$. We say that $(X,f)$ has the \emph{relative property $P$} if there exists a regular periodic decomposition $\mathcal{D}$ such that $(X,f)$ has property $P$ relative to $\mathcal{D}$.
\end{definition}

\begin{theorem}
	Suppose that $f$ is a topologically transitive continuous self-map on a compact interval $X$ with $h_{\mathrm{top}}(f,X)>0$. Then for any non-trivial weight $\vartheta\in \ell^\infty(\mathbb{N})$,
	\[
	h_{\mathrm{top}}(f,\mathcal{N}_\vartheta(f,X))=h_{\mathrm{top}}(f,X)>0.
	\]
	In particular, $(X,f)$ is Bohr chaotic.
\end{theorem}

\begin{proof}
	By \cite[Theorem 7.2]{Banks1997} and its proof, $(X,f)$ is relatively topologically mixing with respect to a regular periodic decomposition
	\[
	\mathcal{D}=\{D_0,\cdots,D_{n-1}\},
	\]
	whose elements are compact intervals. Since $(X,f)$ is topologically transitive, Lemma \ref{Lemma 7.1} implies that
	\[
	D_i=f^i(D_0)\quad \text{for every }0\le i\le n-1,
	\qquad
	X=\bigcup_{i=0}^{n-1}f^i(D_0).
	\]
	Recall that every topologically mixing self-map on a compact interval has the specification property \cite{Blokh-1983}; see also \cite[Appendix A]{Buzzi-1997}. Therefore, $(D_0,f^n|_{D_0})$ satisfies the specification property. Since
	\[
	h_{\mathrm{top}}(f^n,D_0)=n\,h_{\mathrm{top}}(f,X)>0,
	\]
	Theorem \ref{Theorem C} completes the proof.
\end{proof}

\begin{theorem}
	Suppose that $f$ is a topologically transitive continuous self-map on the circle $S^1$ with $h_{\mathrm{top}}(f,S^1)>0$. Then for any non-trivial weight $\vartheta\in \ell^\infty(\mathbb{N})$,
	\[
	h_{\mathrm{top}}(f,\mathcal{N}_\vartheta(f,S^1))=h_{\mathrm{top}}(f,S^1)>0.
	\]
	In particular, $(S^1,f)$ is Bohr chaotic.
\end{theorem}

\begin{proof}
	Since $h_{\mathrm{top}}(f,S^1)>0$, the system $(S^1,f)$ is not conjugate to an irrational rotation. By \cite[Theorem 7.4]{Banks1997} and its proof, $(S^1,f)$ is relatively topologically mixing with respect to a regular periodic decomposition
	\[
	\mathcal{D}=\{D_0,\cdots,D_{n-1}\},
	\]
	whose elements are compact intervals. Since $(S^1,f)$ is topologically transitive, Lemma \ref{Lemma 7.1} implies that
	\[
	D_i=f^i(D_0)\quad \text{for every }0\le i\le n-1,
	\qquad
	S^1=\bigcup_{i=0}^{n-1}f^i(D_0).
	\]
	As above, $(D_0,f^n|_{D_0})$ satisfies the specification property. Since
	\[
	h_{\mathrm{top}}(f^n,D_0)=n\,h_{\mathrm{top}}(f,S^1)>0,
	\]
	Theorem \ref{Theorem C} completes the proof.
\end{proof}

Next, we consider subshifts with positive topological entropy that are factors of an $S$-gap shift or a $\beta$-shift.

\begin{theorem}
	Suppose that a subshift $(X,\sigma)$ with $h_{\mathrm{top}}(\sigma,X)>0$ is a factor of an $S$-gap shift or a $\beta$-shift. Then for any non-trivial weight $\vartheta\in \ell^\infty(\mathbb{N})$,
	\[
	h_{\mathrm{top}}(\sigma,\mathcal{N}_\vartheta(\sigma,X))=h_{\mathrm{top}}(\sigma,X)>0.
	\]
	In particular, $(X,\sigma)$ is Bohr chaotic.
\end{theorem}

\begin{proof}
	By \cite[Theorem B]{Climenhaga-Thompson-Yamamoto-2017}, for any $\eta<h_{\mathrm{top}}(\sigma,X)$, there exists a topologically transitive sofic subshift $X_\eta\subset X$ with $h_{\mathrm{top}}(\sigma,X_\eta)>\eta$. Moreover, by the proof of \cite[Corollary 40]{Kwietniak-Lacka-Oprocha-2016}, the system $(X_\eta,\sigma)$ has the relative specification property with respect to a regular periodic decomposition
	\[
	\mathcal{D}=\{D_0,\cdots,D_{n-1}\}.
	\]
	Since $(X_\eta,\sigma)$ is topologically transitive, Lemma \ref{Lemma 7.1} implies that
	\[
	D_i=\sigma^i(D_0)\quad \text{for every }0\le i\le n-1,
	\qquad
	X_\eta=\bigcup_{i=0}^{n-1}\sigma^i(D_0).
	\]
	Therefore, Theorem \ref{Theorem C} yields
	\[
	h_{\mathrm{top}}(\sigma,\mathcal{N}_\vartheta(\sigma,X))
	\ge
	h_{\mathrm{top}}(\sigma,\mathcal{N}_\vartheta(\sigma,X_\eta))
	>
	\eta.
	\]
	Finally, since $\eta<h_{\mathrm{top}}(\sigma,X)$ is arbitrary, the conclusion follows.
\end{proof}

\bigskip

\textbf{Acknowledgements.}   The authors would like to thank an anonymous referee for his/her careful reading and valuable comments, which helped improve the presentation of this paper. X. Hou is supported by the National Natural Science Foundation of China (No. 12401231), China Postdoctoral Science Foundation (No. 2023M740713), and the Postdoctoral Fellowship Program of CPSF under Grant Number GZB20240167. W. Lin is supported by the National Natural Science Foundation of China (No. 124B2010). X. Tian is supported by the National Natural Science Foundation of China (No. 12471182) and  Natural Science Foundation of Shanghai (No. 23ZR1405800).

\textbf{Conflict of interest.} The authors declare that there is no conflict of interest.

\textbf{Data availability.} No data was used for the research described in the article.

\end{document}